\documentclass[11pt, reqno]{amsart}
\usepackage{amssymb,amsmath,epsfig,mathrsfs, enumerate, xparse, mathtools}
\usepackage[pagewise]{lineno}
\usepackage[nodisplayskipstretch]{setspace}
\usepackage{graphicx}\usepackage[normalem]{ulem}
\usepackage{fancyhdr}
\usepackage[margin=2.5cm]{geometry}
\usepackage{hyperref}
\hypersetup{
colorlinks   = true,
urlcolor     = blue,
linkcolor    = blue,
citecolor   = red ,
bookmarksopen=true
}

\theoremstyle{plain}
\newtheorem{thrm}{Theorem}[section]
\newtheorem{lemma}[thrm]{Lemma}

\newtheorem{rmrk}[thrm]{Remark}

 \usepackage[usenames]{color}

\allowdisplaybreaks
\begin{document}
\newcommand{\sn}{\mathbb{S}^{n-1}}
\newcommand{\SL}{\mathcal L^{1,p}( D)}
\newcommand{\Lp}{L^p( Dega)}
\newcommand{\CO}{C^\infty_0( \Omega)}
\newcommand{\Rn}{\mathbb R^n}
\newcommand{\Rm}{\mathbb R^m}
\newcommand{\R}{\mathbb R}
\newcommand{\Om}{\Omega}
\newcommand{\Hn}{\mathbb H^n}
\newcommand{\A}{\alpha }
\newcommand{\B}{\beta}
\newcommand{\eps}{\ve}
\newcommand{\BVX}{BV_X(\Omega)}
\newcommand{\p}{\partial}
\newcommand{\IO}{\int_\Omega}
\newcommand{\bG}{\boldsymbol{G}}
\newcommand{\bg}{\mathfrak g}
\newcommand{\bz}{\mathfrak z}
\newcommand{\bv}{\mathfrak v}
\newcommand{\Bux}{\mbox{Box}}
\newcommand{\e}{\ve}
\newcommand{\X}{\mathcal X}
\newcommand{\Y}{\mathcal Y}
\newcommand{\Z}{\mathcal Z}
\newcommand{\la}{\lambda}
\newcommand{\vf}{\varphi}
\newcommand{\rhh}{|\nabla_H \rho|}
\newcommand{\Ba}{\mathcal{B}_\beta}
\newcommand{\Za}{Z_\beta}
\newcommand{\ra}{\rho_\beta}
\newcommand{\n}{\nabla}
\newcommand{\vt}{\vartheta}
\newcommand{\its}{\int_{\{y=0\}}}
\newcommand{\py}{\partial_y^a}

\numberwithin{equation}{section}

\newcommand{\RN} {\mathbb{R}^N}
\newcommand{\Sob}{S^{1,p}(\Omega)}
\newcommand{\Dxk}{\frac{\partial}{\partial x_k}}
\newcommand{\Co}{C^\infty_0(\Omega)}
\newcommand{\Je}{J_\ve}
\newcommand{\beq}{\begin{equation}}
\newcommand{\bea}[1]{\begin{array}{#1} }
	\newcommand{\eeq}{ \end{equation}}
\newcommand{\ea}{ \end{array}}
\newcommand{\eh}{\ve h}
\newcommand{\Dxi}{\frac{\partial}{\partial x_{i}}}
\newcommand{\Dyi}{\frac{\partial}{\partial y_{i}}}
\newcommand{\Dt}{\frac{\partial}{\partial t}}
\newcommand{\aBa}{(\alpha+1)B}
\newcommand{\GF}{\psi^{1+\frac{1}{2\alpha}}}
\newcommand{\GS}{\psi^{\frac12}}
\newcommand{\HFF}{\frac{\psi}{\rho}}
\newcommand{\HSS}{\frac{\psi}{\rho}}
\newcommand{\HFS}{\rho\psi^{\frac12-\frac{1}{2\alpha}}}
\newcommand{\HSF}{\frac{\psi^{\frac32+\frac{1}{2\alpha}}}{\rho}}
\newcommand{\AF}{\rho}
\newcommand{\AR}{\rho{\psi}^{\frac{1}{2}+\frac{1}{2\alpha}}}
\newcommand{\PF}{\alpha\frac{\psi}{|x|}}
\newcommand{\PS}{\alpha\frac{\psi}{\rho}}
\newcommand{\ds}{\displaystyle}
\newcommand{\Zt}{{\mathcal Z}^{t}}
\newcommand{\XPSI}{2\alpha\psi \begin{pmatrix} \frac{x}{\left< x \right>^2}\\ 0 \end{pmatrix} - 2\alpha\frac{{\psi}^2}{\rho^2}\begin{pmatrix} x \\ (\alpha +1)|x|^{-\alpha}y \end{pmatrix}}
\newcommand{\ZZ}{ \begin{pmatrix} xx^{t} & (\alpha + 1)|x|^{-\alpha}x y^{t}\\
	(\alpha + 1)|x|^{-\alpha}x^{t} y &   (\alpha + 1)^2  |x|^{-2\alpha}yy^{t}\end{pmatrix}}
\newcommand{\norm}[1]{\lVert#1 \rVert}
\newcommand{\ve}{\varepsilon}
\newcommand{\D}{\operatorname{div}}
\newcommand{\G}{\mathscr{G}}
\newcommand{\w}{\tilde{w}}
\newcommand{\s}{\sigma}

\title[quantitative uniqueness etc.]{ Space-like   Quantitative uniqueness for parabolic operators}

\author{Vedansh Arya}
\address{Tata Institute of Fundamental Research\\
Centre For Applicable Mathematics \\ Bangalore-560065, India}\email[Vedansh Arya]{vedansh@tifrbng.res.in}

\author{Agnid Banerjee}
\address{Tata Institute of Fundamental Research\\
Centre For Applicable Mathematics \\ Bangalore-560065, India}\email[Agnid Banerjee]{agnidban@gmail.com}

\thanks{A.B is supported  by Department of Atomic Energy,  Government of India, under
project no.  12-R \& D-TFR-5.01-0520.}


%
%
%

\maketitle
\begin{abstract}
We obtain sharp maximal  vanishing order     at a given time level  for   solutions  to parabolic equations  with a $C{^1}$ potential $V$.  Our main result Theorem \ref{main}  is  a parabolic generalization of a well known result of Donnelly-Fefferman and Bakri.  It also sharpens   a previous result of Zhu that establishes similar vanishing order estimates which are instead averaged over time.  The principal tool in our analysis is a   new quantitative version of the well-known Escauriaza-Fernandez-Vessella type Carleman estimate that we establish in our setting.

\end{abstract}

\tableofcontents

\section{Introduction and the Statement of the main result}
In this paper, we study quantitative uniqueness for non-trivial solutions to 

\begin{align}\label{meq}
	u_t+\D(A(x,t)\n u) +V(x,t)u=0,
\end{align}
where $V(x,t)$ is $C^1$ in $x$-variable and $1/2$-H\"{o}lder continuous in $t$-variable and the matrix valued function  $A(x,t)$ satisfies the following ellipticity and growth condition
\begin{equation}\label{assump}
\begin{cases}
\Lambda^{-1} \mathbb I_n \leq A\leq \Lambda \mathbb I_n,\ \Lambda \geq1\\
|A(x,t) - A(y,s) |\leq K(|x-y| + |t-s|^{1/2}).\end{cases}\end{equation}
We first review some basic results  about  quantitative uniqueness for elliptic equations. 
The vanishing order of a function $u$ at $x_0$ is the largest  integer $\ell$ at   such that $D^{\alpha} u=0$ for all $|\alpha| \leq \ell$, where $\alpha$ is a multi-index. In the papers \cite{DF1}, \cite{DF2}, Donnelly and Fefferman showed that if $u$ is an eigenfunction with eigenvalue  $\lambda$ on a  smooth, compact and connected $n$-dimensional Riemannian manifold $M$, then  the maximal vanishing order of $u$ is less than $C \sqrt{\lambda},$ where $C$ only depends on the manifold $M$. Using this estimate, they showed  that $H^{n-1}(\{x: u_{\lambda}(x)=0\})\leq C  \sqrt{\lambda},$ where $u_{\lambda}$ is the eigenfunction corresponding to $\lambda$ and therefore  gave a complete  answer to a famous  conjecture of Yau (\cite{Yau}). We note that the zero set of $u_{\lambda}$ is referred to as the nodal set. This order of vanishing is sharp. If, in fact, we consider $M = \mathbb S^n \subset \R^{n+1}$, and we take the spherical harmonic $Y_\kappa$ given by the restriction to $\mathbb S^n$ of the function $f(x_1,...,x_n,x_{n+1}) = \Re (x_1 + i x_2)^\kappa$, then one has $\Delta_{\mathbb S^n} Y_\kappa = - \lambda_\kappa Y_\kappa$, with $\lambda_\kappa = \kappa(\kappa+n-2)$, and the order of vanishing of $Y_\kappa$ at the North pole $(0,...,0,1)$ is precisely $\kappa = C \sqrt {\lambda_\kappa}$. 

Kukavica in   \cite{Ku}  considered the more general problem
\begin{equation}\label{e1}
\Delta u = V(x) u,
\end{equation}
where $V\in W^{1, \infty}$, and showed that the maximal vanishing order  of $u$ is bounded above by $C( 1+ ||V||_{W^{1, \infty}})$. He also conjectured that the rate of vanishing order of $u$ is less than or equal to $C(1+ ||V||_{L^{\infty}}^{1/2})$, which agrees with the Donnelly-Fefferman result when $V = - \lambda$. Employing Carleman estimates,  Bourgain and Kenig in \cite{BK}  ( see also \cite{EV}) showed that the rate of  vanishing order  of $u$  is less than $C(1+ ||V||_{L^{\infty}}^{2/3})$, and furthermore the exponent $\frac{2}{3}$ is sharp for complex  potentials $V$ based on a counterexample of Meshov (see \cite{Me}).

Not so long ago,  for equations of the type \eqref{e1}, the rate of vanishing order of $u$ has been shown to be less than $C(1+ ||V||_{W^{1, \infty}}^{1/2})$  independently by Bakri in \cite{Bk} and Zhu in \cite{Zhu1}. Bakri's approach is based on  an extension of the Carleman method in \cite{DF1}.  On the other hand, Zhu's approach is based on a variant of the frequency function approach  employed by Garofalo and Lin in \cite{GL1}, \cite{GL2}), in the context of strong  unique  continuation problems. The approach of Zhu has been  subsequently extended in \cite{BG1}  to variable coefficient principal part with Lipschitz coefficients  where a similar quantitative  uniqueness result at the boundary of $C^{1, Dini}$ domains has been obtained.

Very recently, Zhu in \cite{Zhu2} showed that a nontrivial solution $u$  to \eqref{meq}  in $B_1 \times (0,1)$  satisfies the following quantitative uniqueness estimate for all $r$ small

\begin{equation}\label{ert}
\int_{B_r \times (0,1)}  u(x,t)^2 dxdt \geq C r^{C_1 ||V||_{C^{1}_{x,t}}^{1/2} +C_2},
\end{equation}
when the coefficient matrix  $A$ and the potential $V$ are both Lipschitz in space and time. This  generalizes the result  in \cite{Bk} and \cite{Zhu1} to the parabolic case. This was done by derivation  of  an appropriate quantitative  version of Vessella's Carleman estimate as in \cite{Ve}.  A nonlocal generalization  of  Zhu's result   has been recently obtained by both of us in \cite{AB}.  See also \cite{BL, Ru, Zhu0} for  similar results that were previously established in the fractional elliptic setting. 

\medskip

We now  observe that since the  estimate in \eqref{ert} is on space time cylinders of the type $B_r \times (0,1)$, it doesn't give any vanishing order information  for the solution $u$  at  a given time level, say $t=0$. 
  At this point, we    recall  that  the   space-like  strong unique continuation results for equations of the type \eqref{meq}  states  that if  a solution $u$ to \eqref{meq} vanishes to infinite order in $x$ at $(0,0)$ where $A$ satisfies \eqref{assump}, then  $u(\cdot,  0) \equiv 0$. For such results, we refer to the well known works \cite{EF, EFV}   and also to  \cite{KT}  where coefficients with lower regularity assumptions are considered. Therefore in the spirit of such space-like strong unique continuation results, it is reasonable to expect that a  solution $u$ to \eqref{meq} satisfies an estimate   of the type  
\begin{equation}\label{ert1}
\int_{B_r}  u(x, 0)^2 dx \geq C r^{C_1 ||V||_{C^{1}_{x,t}}^{1/2} +C_2},
\end{equation}
provided $u(\cdot, 0) \not\equiv 0$.  This is precisely the content of our main result which we now state. We refer to Section \ref{s:n} for the precise notations.
\begin{thrm}\label{main}
	Let $u$ be a solution to \eqref{meq} in $Q_4$ such that $u(\cdot, 0) \not \equiv  0$ in $B_1$. Then there exists a universal constant $N$ such that for all $r \le 1/2$,  one has 
	\begin{align}\label{df}
		\int_{B_r} u^2(x,0)dx \geq   r^{\mathcal{C}},
\end{align}
where $\mathcal{C}=\frac{1}{\int_{B_1} u^2(x,0)dx}+ N\operatorname{log}(N\Theta)+N(||V||_1^{1/2}+[V]_{1/2}^{1/2}+1)$, $\Theta=	\frac{N\int_{Q_{4}} u^2( x, t)dxdt}{\int_{B_{1}}u^2(x,0)dx}$, 
$$||V||_1\overset{def}{=}||V||_{L^{\infty}(Q_4)} +||\n_x V||_{L^{\infty}(Q_4)} $$ and $$[V]_{1/2}\overset{def}{=}{\operatorname{sup}}_{(x,t), (x,s) \in Q_4} \frac{|V(x,t)-V(x,s)|}{|t-s|^{1/2}}.$$ 
\end{thrm}
It is easily seen that our space like non degeneracy  estimate  \eqref{df} refines the space time estimate \eqref{ert}. Moreover,  Theorem \ref{main} also requires lower regularity assumptions on the principal part $A$ and the potential $V$  as compared to that in \cite{Zhu2}. 

The key ingredient in the proof of Theorem \ref{main} is a certain quantitative version of the Carleman estimate in \cite{EFV} that we establish in our setting. See the estimate in  \eqref{carl} below. This is the key novelty of our work. We would like to mention over here that although  the  proof of our Carleman estimate borrows certain essential  ideas from the works  \cite{EF, EFV}, at the same time it relies on a somewhat different conjugation method   that appears in a very  recent work of one of us with Krishnan and Senapati in \cite{BKS} in the context of fractional heat inverse problem.  Having said that,  in the present work since we are concerned about quantitative uniqueness, therefore in our Carleman estimate \eqref{carl},  we additionally need  to ensure that  the "vanishing order" parameter $\alpha$ depends  quantitatively   on the $C^{1}$ norm of  the potential $V$ in an appropriate way.  This in fact entails some  delicate adaptations of the method in \cite{BKS}  to our setting.  In closing, we refer to the works \cite{CK} and \cite{KL} for  other variants of the  quantitative uniqueness results in the parabolic setting.

The paper is organized as follows. In section \ref{s:n}, we introduce some basic notations and gather some  known results that are relevant to our work. In section \ref{s:main}, we prove our main result.

\section{Preliminaries and Notations}\label{s:n}
A generic point in space time $\Rn \times [0, \infty)$ will be denoted by the variables $X=(x,t), Y=(y,s)$, etc. For notational convenience, $\nabla f$ and  $\operatorname{div}\ f$ will respectively refer to the quantities  $\nabla_x f$ and $ \operatorname{div}_x f$.   The partial derivative in $t$ will be denoted by $\p_t f$ and also by $f_t$. The partial derivative $\partial_{x_i} f$  will be denoted by $f_i$. We indicate with $C_0^{\infty}(\Omega)$ the set of compactly supported smooth functions in the region $\Omega$  in space-time. $dX$ will be denoted  by $dxdt.$ $B_r(x)$ will denote a  ball of radius $r$ with centre at $x \in \R^n$, $Q_r(x,t)$ will denote the space time  cylinder $B_r(x) \times [t,t+r^2].$  For further notational convenience, we will denote $B_r(0)$ and $Q_r(0,0)$ by $B_r$ and $Q_r$ respectively.

We now state some preparatory  results that is needed in the present work.   The following lemma which is Lemma 4 in \cite{EF} is regarding the existence of a suitable weight function $\sigma$ which has the appropriate convexity property for our Carleman estimate. \begin{lemma}\label{sig}
	Consider $$\theta(t)=t^{1/2}\left( \operatorname{log}\frac{1}{t}\right)^{3/2}.$$
Then, the solution of the ordinary differential equation
	\begin{align}\label{ode}
		\frac{d}{dt}\operatorname{log} \left(\frac{\s}{t\s'}\right)=\frac{\theta(\lambda t)}{t},\;\;\;\;\s(0)=0,\;\;\s'(0)=1,
	\end{align}
where $\lambda >0$ and has the following proprties  when $0\le \lambda t \le 1$
\begin{itemize}
	\item[(i)] $t/N \le \s(t) \le t.$
	\item[(ii)]  $1/N \le \s'(t) \le 1.$
\end{itemize}
where $N$ is a universal constant.
\end{lemma}
The next lemma which is Lemma 5 in \cite{EF} is regarding certain estimates in the Gaussian space. This  will be needed to handle some error terms in the proof of our Carleman estimate that arises due to Lipschitz  perturbation of the principal part.
\begin{lemma}\label{logi}
	Let $G(x,t)=t^{-n/2}e^{-|x|^2/4t}$ and $\s$ denotes the function defined in Lemma \ref{sig} for $\lambda = \A/ \delta^2$ and $$\theta(t)=t^{1/2}\left( \operatorname{log}\frac{1}{t}\right)^{3/2}.$$ Then, there exists a constant $N$ depending only on $n$ such that the following inequalities hold for all functions $w \in C_0^{\infty}(\R^n \times [0,1/2\lambda))$,
	\begin{align*}
	\int	\s^{-2\A}\left(\frac{|x|}{t}+\frac{|x|^3}{\A t^2}\right)w^2G\le NN^{2\A}\lambda^{2\A+N}\int w^2+ N \delta \int \s^{-2\A}\frac{\theta(\lambda t)}{t}w^2 G
	\end{align*}
and 
\begin{align*}
\int	\s^{1-2\A}\left(\frac{|x|}{t}+\frac{|x|^3}{\A t^2}+\frac{|x|^2}{\delta t}\right)|\n w|^2G\le NN^{2\A}\lambda^{2\A+N}\int t|\n w|^2+ N \delta \int \s^{1-2\A}\frac{\theta(\lambda t)}{t}|\n w|^2 G.
\end{align*}
\end{lemma}
We also need the following  lemma ( see \cite[Lemma 3]{EFV} ) that  will be used to estimate certain boundary terms in the proof of  our Carleman estimate.
\begin{lemma}\label{logb}
	For all $h \in C_0^{\infty}(\R^n)$ and $a>0$ the following inequality holds
	$$\int \frac{|x|^2}{8a}h^2 e^{-|x|^2/4a}dx \le 2a\int |\n h|^2  e^{-|x|^2/4a}dx +\frac{n}{2}\int h^2  e^{-|x|^2/4a}dx.$$
\end{lemma}
The next  lemma  which is \cite[Lemma 4]{EFV} will be eventually used to obtain a quantitative  space like doubling property.
\begin{lemma}\label{do}
	Assume that $N \ge 1,$ $h \in C_{0}^{\infty}(\R^n)$ and the inequality \begin{align*}
		2a \int |\n h|^2 e^{-|x|^2/4a}dx + \frac{n}{2}\int h^2 e^{-|x|^2/4a}dx \le N  \int h^2  e^{-|x|^2/4a}dx
	\end{align*}
holds for $a \le \frac{1}{12N }.$ Then
\begin{align}
	\int_{B_{2r}}h^2 dx \le e^N  \int_{B_{r}}h^2 dx 
\end{align}
when $0 < r \le 1/2.$
\end{lemma}
We will also  need the following regularity estimates for solutions to \eqref{meq} which can be found in \cite[Chapter 6]{Li}.
\begin{lemma}\label{enl}
	Let $u$ be a solution of \eqref{meq}. Then there exists a universal constant $C_E$ depending on the Lipschitz character of the coefficient matrix $A$ and $n$  such that 
	\begin{align}\label{eniq}
		||u||_{L^{\infty}(Q_3)}+||\n u||_{L^{\infty}(Q_3)} \le C_E(1+||V||_{\infty})|| u||_{L^2(Q_4)}.
	\end{align}
\end{lemma}
In  closing we make the following discursive but  important remark.

\begin{rmrk}
	In the rest of the work, whenever we    say  that a constant $N$  is  universal, it means that  it depends only on the dimension $n$ and the ellipticity and the  Lipschitz norm of  $A.$ Throughout the paper we will use $N$ as all purpose constant which may vary from line to line but will be universal.
\end{rmrk}

\section{Proof of Theorem \ref{main}}\label{s:main}

The following Carleman estimate below is one of the key ingredients in the proof of Theorem \ref{main}.  This is analogous to Lemma 6 in \cite{EFV}, however   the key new feature being the asymptotic dependence  of the weight parameter $\alpha$ on the "$C^{1}$" norm of $V$.   Such an asymptotic dependence is then  crucially exploited in the proof of  the quantitative estimate  \eqref{df} in Theorem \ref{main}.
\begin{lemma}\label{carleman}
	Let $A(0,0)=\mathbb I_n$. There exists a universal constant $N$ and $\delta \in (0,1)$ such that for all $\A \ge N(1+||V||_1^{1/2}+[V]_{1/2}^{1/2}),$ the following inequality holds
	\begin{align}\label{carl}
		&	\A^2 \int \s_a^{-2\A}w^2G_a +\A \int \s_a^{1-2\A}|\n w|^2G_a\\
		& \le N \int \s_a^{1-2\A}(\D(A(x,t) \n w)+w_t+V(x,t)w)^2 G_a +N^{2\A}\A^{2\A}\underset{t \ge 0}{\operatorname{sup}}\int w^2
		+t|\n w|^2 dx\notag\\
		&+\s(a)^{-2\A}\left(-\frac{a}{N}\int |\n w (x,0)|^2G(x,a)dx + N \A \int w^2(x,0)G(x,a)dx\right).\notag
	\end{align} for all $0<a\le \frac{1}{4\lambda}$ and $w \in C_{0}^{\infty}(B_4 \times [0,\frac{1}{2\lambda}))$, where $\lambda=\A/\delta^2,$ $G_a(x,t)=G(x,t+a)$ and $\s_a(t)=\s(t+a).$
\end{lemma}
\begin{proof}
The proof is divided into several steps.

\textbf{Step 1}: We first show that there exists a universal constant $N$ and $\delta \in (0,1)$ such that for all $\A \ge N(1+||V(\cdot, 0)||_1^{1/2})$ the following inequality holds
	\begin{align}\label{step1}
		&	\A^2 \int \s^{-2\A}w^2G + \int \s^{1-2\A}|\n w|^2\frac{\theta(\lambda t)}{t}G\\
		& \le N \int \s^{1-2\A}(\D(A(x,0) \n w)+w_t+V(x,0)w)^2 G +N^{2\A}\lambda^{2\A+N}\underset{t \ge a}{\operatorname{sup}}\int w^2
		+t|\n w|^2 dx\notag\\
		&+\s(a)^{-2\A}\left(-\frac{a}{N}\int |\n w (x,a)|^2G(x,a)dx + N \A \int w^2(x,a)G(x,a)dx\right).\notag
	\end{align}
	 for all $0<a\le \frac{1}{4\lambda}$ and $w \in C_{0}^{\infty}(B_4 \times [0,\frac{1}{2\lambda}))$, where $\lambda=\A/\delta^2.$

	 \medskip

\textbf{	Proof of Step 1}: For notational convenience, we will denote $A(x,0)$ and $V(x,0)$ and with slight abuse of notation  by $A$ or $A(x)$ and $V$ or $V(x).$ respectively.  All solid integrals will be supported in $B_4 \times [a,\frac{1}{2\lambda}].$ We will refrain from writing it. We set $\w(x,t) =\s^{-\A}e^{-\frac{|x|^2}{8t}}w(x,t),$ where $\s$ is as defined in Lemma \ref{sig}. Then $w(x,t)=\s^{\A}e^{\frac{|x|^2}{8t}}\w(x,t)$ and we have
	\begin{align}
		w_t=\A\s^{\A-1}\s'e^{\frac{|x|^2}{8t}}\w -\frac{|x|^2}{8t^2}\s^{\A}e^{\frac{|x|^2}{8t}}\w+ \s^{\A}e^{\frac{|x|^2}{8t}}\w_t
	\end{align}
and
\begin{align}
	\D(A \n w)&= \s^{\A}e^{\frac{|x|^2}{8t}} \D( A \n \w) + 2\s^{\A}e^{\frac{|x|^2}{8t}} \left\langle A \n \w, \frac{x}{4t} \right\rangle
	+ \s^{\A}e^{\frac{|x|^2}{8t}} \D \left(\frac{Ax}{4t}\right)  \w
	 +\s^{\A}e^{\frac{|x|^2}{8t}} \left\langle \frac{Ax}{4t}, \frac{x}{4t}\right\rangle.\notag
\end{align}
Hence we have 
\begin{align}\label{1st}
	&\int \s^{-2\A}e^{\frac{-|x|^2}{4t}}\left( w_t + \D(A \n w) + V w\right)^2t^{-\mu}\left(\frac{t \s'}{\s}\right)^{-1/2}\\
	&= \int \left( \frac{\A\s'}{\s}\w-\frac{|x|^2}{8t^2}\w+\w_t + \D(A \n \w)+ \left\langle A \n \w, \frac{x}{2t}\right\rangle +\frac{\D(Ax)}{4t}\w + \left\langle \frac{Ax}{4t}, \frac{x}{4t}\right\rangle\w + V \w \right)^2 t^{-\mu}\left(\frac{t \s'}{\s}\right)^{-1/2},
\end{align}
where $\mu=\frac{n}{2}-1.$
We now define $\Z\w:=2t\w_t + \langle A \n \w, x\rangle.$ We now use the numerical identity $(a+b)^2 \ge a^2 +2 ab$ in \eqref{1st} with $a= \Z\w$ and $b$  being the  remaining terms to find
\begin{align}\label{amgm}
	&\int \s^{-2\A}e^{\frac{-|x|^2}{4t}}\left( w_t + \D(A \n w) + V w\right)^2t^{-\mu}\left(\frac{t \s'}{\s}\right)^{-1/2}\\
	 &\ge \int \left(\frac{\Z\w}{2t}\right)^2t^{-\mu}\left(\frac{t \s'}{\s}\right)^{-1/2} + 2 \int \frac{\A \s'}{\s}\w\frac{\Z \w}{2t} t^{-\mu}\left(\frac{t \s'}{\s}\right)^{-1/2}\notag\\
	 &-2\int \frac{|x|^2}{8t^2}\w\frac{\Z \w}{2t} t^{-\mu}\left(\frac{t \s'}{\s}\right)^{-1/2} + 2\int \frac{\D(Ax)}{4t}\w\frac{\Z \w}{2t} t^{-\mu}\left(\frac{t \s'}{\s}\right)^{-1/2}+ 2\int \left\langle \frac{Ax}{4t}, \frac{x}{4t}\right\rangle\w \frac{\Z \w}{2t}t^{-\mu}\left(\frac{t \s'}{\s}\right)^{-1/2}\notag\\
	 &+2 \int \D(A\n\w) \frac{\Z \w}{2t}t^{-\mu}\left(\frac{t \s'}{\s}\right)^{-1/2} + 2 \int V\w \frac{\Z \w}{2t}t^{-\mu}\left(\frac{t \s'}{\s}\right)^{-1/2}\notag\\
	 &=I_1+I_2+I_3+I_4+I_5+I_6+I_7.\notag
\end{align}
We now simplify each term individually. We start with $I_2.$ Using $2\w\Z\w=\Z\w^2$ and after writing $ \frac{\A \s'}{\s}=\frac{\A}{t} \frac{t \s'}{\s},$ we get
\begin{align*}
	I_2&=2 \int \frac{\A \s'}{\s}\w\frac{\Z \w}{2t} t^{-\mu}\left(\frac{t \s'}{\s}\right)^{-1/2}=\A\int\frac{1}{t} \frac{t \s'}{\s}\frac{\Z \w^2}{2t} t^{-\mu}\left(\frac{t \s'}{\s}\right)^{-1/2}\\
	&=\A\int\left((\w^2)_t + \left\langle A\n\w^2,\frac{x}{2t}\right\rangle\right) t^{-\mu-1}\left(\frac{t \s'}{\s}\right)^{1/2}.
\end{align*}
We now use the divergence theorem to obtain 
\begin{align*}
	I_2=-\A\int\w^2 \left(t^{-\mu-1}\left(\frac{t \s'}{\s}\right)^{1/2}\right)' -\A\int_{\{t=a\}}\w^2 t^{-\mu-1}\left(\frac{t \s'}{\s}\right)^{1/2}-\frac{\A}{2}\int \D(Ax)\w^2 t^{-\mu-2}\left(\frac{t \s'}{\s}\right)^{1/2}.
\end{align*}
Since $A$ is Lipschitz in $x$-variable  and $A(0,0)= \mathbb I_n$, we have $\D(Ax)=n+O(|x|).$ Hence after simplification we find
\begin{align*}
	I_2&=\A(\mu+1)\int t^{-\mu-2}\left(\frac{t \s'}{\s}\right)^{1/2} \w^2 -\frac{\A}{2}\int t^{-\mu-1}\left(\frac{t \s'}{\s}\right)^{-1/2}\left(\frac{t \s'}{\s}\right)'\w^2-\A\int_{\{t=a\}}t^{-\mu-1}\left(\frac{t \s'}{\s}\right)^{1/2} \w^2\\
	& -\frac{\A n}{2}\int t^{-\mu-2}\left(\frac{t \s'}{\s}\right)^{1/2} \w^2 +\A \int O(|x|)t^{-\mu-2}\left(\frac{t \s'}{\s}\right)^{1/2} \w^2.
\end{align*}
We now recall $\mu=\frac{n}{2}-1$ and use \eqref{ode} to find 
\begin{align}\label{i22}
	I_2&= \frac{\A}{2}\int t^{-n/2}\left(\frac{t \s'}{\s}\right)^{1/2}\frac{\theta (\lambda t) }{t}\w^2-\A\int_{\{t=a\}}t^{-n/2}\left(\frac{t \s'}{\s}\right)^{1/2} \w^2
	 +\A \int O(|x|)t^{-n/2-1}\left(\frac{t \s'}{\s}\right)^{1/2} \w^2.
\end{align}
To estimate the last term in the right-hand side of the inequality \eqref{i22}, we recall $\w= \s^{-\A}e^{-|x|^2/8t}w$ and get for some universal $N$
\begin{align*}
	\left|\A \int O(|x|)t^{-n/2-1}\left(\frac{t \s'}{\s}\right)^{1/2} \w^2\right| \le \A N \int \frac{|x|}{t}\left(\frac{t \s'}{\s}\right)^{1/2} w^2\s^{-2\A}G.
\end{align*}
We now use the Lemma \ref{logi} and the fact that $\left(\frac{t \s'}{\s}\right)^{1/2}$ is bounded from above and below  by positive constants to   obtain 
\begin{align}\label{i2l}
\alpha	\int \frac{|x|}{t}\left(\frac{t \s'}{\s}\right)^{1/2} w^2\s^{-2\A}G \le N N^{2\A}\lambda^{2\A+N} \int w^2 +N\delta \A \int \s^{-2\A}\frac{\theta(\lambda t)}{t}w^2G.
\end{align}
Substitute the bound from \eqref{i2l} in \eqref{i22} and write $\w$ in terms of $w$ to find 
\begin{align}\label{i2ff}
		I_2&\ge \frac{\A}{2}\int \s^{-2\A}\left(\frac{t \s'}{\s}\right)^{1/2}\frac{\theta(\lambda t) }{t}w^2G-\A\int_{\{t=a\}}\s^{-2\A}\left(\frac{t \s'}{\s}\right)^{1/2} w^2 G\\
		&-N N^{2\A}\lambda^{2\A+N} \int w^2 -N\delta \A \int \s^{-2\A}\frac{\theta(\lambda t)}{t}w^2G.
\end{align}
Next we consider $I_3.$ Using  divergence theorem we find
\begin{align}\label{i31}
	I_3=-2\int \frac{|x|^2}{8t^2}\w\frac{\Z \w}{2t} t^{-\mu}\left(\frac{t \s'}{\s}\right)^{-1/2} &=-\int \frac{|x|^2}{8t^2}\frac{\Z \w^2}{2t} t^{-\mu}\left(\frac{t \s'}{\s}\right)^{-1/2}\\
	&=\int \frac{|x|^2}{8} \left(t^{-\mu-2}\left(\frac{t \s'}{\s}\right)^{-\frac{1}{2}}\right)'\w^2+\int_{\{t=a\}} \frac{|x|^2}{8} t^{-\mu-2}\left(\frac{t \s'}{\s}\right)^{-\frac{1}{2}}\w^2\notag\\
	&+\int \frac{\D(|x|^2Ax)}{16}t^{-\mu-3}\left(\frac{t \s'}{\s}\right)^{-\frac{1}{2}}\w^2.\notag
\end{align}
Again using $A(x)$ is Lipschitz in $x$-variable and $A(0,0)=\mathbb I_n$,  we have $\D(|x|^2 Ax)= (n+2)|x|^2 +|x|^2O(|x|)$ and  thus after simplifying the first term in right hand side of \eqref{i31} we obtain
\begin{align}\label{i32}
	I_3&=\frac{-\mu-2}{8}\int |x|^2 t^{-\mu-3}\left(\frac{t \s'}{\s}\right)^{-\frac{1}{2}}\w^2-\frac{1}{16}\int |x|^2 t^{-\mu-2}\left(\frac{t \s'}{\s}\right)^{-\frac{3}{2}}\left(\frac{t \s'}{\s}\right)'\w^2+\int_{\{t=a\}} \frac{|x|^2}{8} t^{-\mu-2}\left(\frac{t \s'}{\s}\right)^{-\frac{1}{2}}\w^2\\
	&+\frac{n+2}{16}\int |x|^2 t^{-\mu-3}\left(\frac{t \s'}{\s}\right)^{-\frac{1}{2}}\w^2 +O(1)\int |x|^3 t^{-\mu-3}\left(\frac{t \s'}{\s}\right)^{-\frac{1}{2}}\w^2.\notag
\end{align}
Since $\mu =\frac{n}{2}-1$, therefore $\frac{-\mu-2}{8}=-\frac{n+2}{16}.$  Hence after using \eqref{ode}, \eqref{i32} becomes
\begin{align}\label{i33}
		I_3&=\frac{1}{16}\int |x|^2 t^{-n/2-1}\left(\frac{t \s'}{\s}\right)^{-\frac{1}{2}}\frac{\theta(\lambda t)}{t}\w^2+\int_{\{t=a\}} \frac{|x|^2}{8} t^{-\mu-2}\left(\frac{t \s'}{\s}\right)^{-\frac{1}{2}}\w^2\\
	&+O(1)\int \frac{|x|^3}{t^2} t^{-n/2}\left(\frac{t \s'}{\s}\right)^{-\frac{1}{2}}\w^2.\notag
\end{align}
We estimate the last term of right-hand side of \eqref{i33} using Lemma \ref{logi} in the similar way as we did the last term in the right-hand side of the inequality \eqref{i22} to find
\begin{align}\label{i3l}
	\left|O(1)\int \frac{|x|^3}{t^2} t^{-n/2}\left(\frac{t \s'}{\s}\right)^{-\frac{1}{2}}\w^2\right| \le N \A N^{2\A}\lambda^{2\A+N} \int w^2 +N\A\delta  \int \s^{-2\A}\frac{\theta(\lambda t)}{t}w^2G \left(\frac{t \s'}{\s}\right)^{1/2}.
\end{align}
Hence using \eqref{i3l} in \eqref{i33} and writing $\w$ in term of $w$, we get
\begin{align}\label{i3ff}
		I_3& \ge \frac{1}{16}\int \frac{|x|^2}{t} \s^{-2\A}\left(\frac{t \s'}{\s}\right)^{-\frac{1}{2}}\frac{\theta(\lambda t)}{t}w^2 G+\int_{\{t=a\}} \frac{|x|^2}{8t} \s^{-2\A}\left(\frac{t \s'}{\s}\right)^{-\frac{1}{2}}w^2 G\\
		 &-N\A N^{2\A}\lambda^{2\A+N} \int w^2 -N\A\delta  \int \s^{-2\A}\frac{\theta(\lambda t)}{t}w^2G .\notag
\end{align}
Next we estimate $I_4.$ We use $\D(Ax)=n+O(|x|)$ to write $I_4$ as follows
\begin{align}\label{i41}
	I_4= 2\int \frac{\D(Ax)}{4t}\w\frac{\Z \w}{2t} t^{-\mu}\left(\frac{t \s'}{\s}\right)^{-1/2}= 2\int \frac{n}{4t}\w\frac{\Z \w}{2t} t^{-\mu}\left(\frac{t \s'}{\s}\right)^{-1/2} +  O(1)\int \frac{|x|}{4t}\w\frac{\Z \w}{2t} t^{-\mu}\left(\frac{t \s'}{\s}\right)^{-1/2}.
\end{align} 
First we focus on the first term of the right hand side of \eqref{i41}. After using $2\w \Z \w = \Z \w^2$ we find
\begin{align}\label{i4a}
	2\int \frac{n}{4t}\w\frac{\Z \w}{2t} t^{-\mu}\left(\frac{t \s'}{\s}\right)^{-1/2}= \frac{n}{4}\int\left((\w^2)_t+\langle A \n \w^2, \frac{x}{2t}\rangle\right) t^{-\mu-1}\left(\frac{t \s'}{\s}\right)^{-1/2}.
\end{align}
We now use integration by parts formula in \eqref{i4a}  to obtain 
\begin{align}\label{4a1}
	2\int \frac{n}{4t}\w\frac{\Z \w}{2t} t^{-\mu}\left(\frac{t \s'}{\s}\right)^{-1/2}&=-\frac{n}{4}\int\left(t^{-\mu-1}\left(\frac{t \s'}{\s}\right)^{-1/2}\right)'\w^2 -\frac{n}{4}\int_{\{t=a\}} t^{-\mu-1}\left(\frac{t \s'}{\s}\right)^{-1/2}\w^2\\ &-\frac{n}{8}\int \D(Ax) \w^2 t^{-\mu-2}\left(\frac{t \s'}{\s}\right)^{-1/2}.\notag
\end{align}
We further simplify the  first term in right-hand side of \eqref{4a1} and use $\D(Ax)=n+O(|x|)$ to get
\begin{align}\label{i45}
	2\int \frac{n}{4t}\w\frac{\Z \w}{2t} t^{-\mu}\left(\frac{t \s'}{\s}\right)^{-1/2}&=\frac{n}{4}(\mu+1)\int t^{-\mu-2}\left(\frac{t \s'}{\s}\right)^{-1/2}\w^2 +\frac{n}{8}\int t^{-\mu-1}\left(\frac{t \s'}{\s}\right)^{-3/2}\left(\frac{t \s'}{\s}\right)'\w^2\\
	&-\frac{n}{4}\int_{\{t=a\}} t^{-\mu-1}\left(\frac{t \s'}{\s}\right)^{-1/2}\w^2-\frac{n^2}{8}\int \w^2 t^{-\mu-2}\left(\frac{t \s'}{\s}\right)^{-1/2}\notag\\
	&+O(1)\int |x|\w^2 t^{-\mu-2}\left(\frac{t \s'}{\s}\right)^{-1/2}.\notag
\end{align} 
Since $\mu =\frac{n}{2}-1,$ we have $\frac{n}{4}(\mu +1)= \frac{n^2}{8}.$ Also, we use \eqref{ode} in the second term of the right-hand side of \eqref{i45}.  Hence \eqref{i41} becomes
\begin{align}\label{i42}
		I_4&\ge -\frac{n}{8}\int t^{-n/2}\left(\frac{t \s'}{\s}\right)^{-1/2}\frac{\theta(\lambda t)}{t}\w^2
		-\frac{n}{4}\int_{\{t=a\}} t^{-n/2}\left(\frac{t \s'}{\s}\right)^{-1/2}\w^2\\&-N\int\frac{|x|}{t}\w^2 t^{-n/2}\left(\frac{t \s'}{\s}\right)^{-1/2}
		-N\int \frac{|x|}{4t}|\w|\frac{|\Z \w|}{2t} t^{-\mu}\left(\frac{t \s'}{\s}\right)^{-1/2}.\notag
\end{align} 
We now use Cauchy-Schwarz inequality in the last term of right-hand side of \eqref{i42} to get 
\begin{align}\label{i4cs}
	N\int \frac{|x|}{4t}\w\frac{\Z \w}{2t} t^{-\mu}\left(\frac{t \s'}{\s}\right)^{-1/2} \le N^2\int \frac{|x|^2}{16t^2}\w^2 t^{-\mu}\left(\frac{t \s'}{\s}\right)^{-1/2}+\frac{1}{4}\int \left(\frac{\Z \w}{2t}\right)^2 t^{-\mu}\left(\frac{t \s'}{\s}\right)^{-1/2}.
\end{align}
Using \eqref{i4cs} in \eqref{i42} we find
\begin{align}\label{i43}
		I_4&\ge -\frac{n}{8}\int t^{-n/2}\left(\frac{t \s'}{\s}\right)^{-1/2}\frac{\theta(\lambda t)}{t}\w^2
	-\frac{n}{4}\int_{\{t=a\}} t^{-n/2}\left(\frac{t \s'}{\s}\right)^{-1/2}\w^2\\&-N\int\frac{|x|}{t}\w^2 t^{-n/2}\left(\frac{t \s'}{\s}\right)^{-1/2}
-N\int \frac{|x|^2}{16t^2}\w^2 t^{-\mu}\left(\frac{t \s'}{\s}\right)^{-1/2}-\frac{1}{4}\int \left(\frac{\Z \w}{2t}\right)^2 t^{-\mu}\left(\frac{t \s'}{\s}\right)^{-1/2}.\notag
\end{align}
After writing $\w=\s^{-\A}e^{-|x|^2/8t}w$ and using Lemma \ref{logi} in the third and fourth term of right-hand side of \eqref{i43} we obtain 
\begin{align}\label{i4ff}
		I_4&\ge -\frac{n}{8}\int \s^{-2\A}\left(\frac{t \s'}{\s}\right)^{-1/2}\frac{\theta(\lambda t)}{t}w^2G
	-\frac{n}{4}\int_{\{t=a\}} \s^{-2\A}\left(\frac{t \s'}{\s}\right)^{-1/2}w^2G\\&- NN^{2\A}\lambda^{2\A+N}\int w^2- N \delta \int \s^{-2\A}\frac{\theta(\lambda t)}{t}w^2 G -\frac{1}{4}\int \left(\frac{\Z \w}{2t}\right)^2 t^{-\mu}\left(\frac{t \s'}{\s}\right)^{-1/2}.\notag
\end{align}
Over here, we would like to mention that In order to estimate the fourth term in \eqref{i43} above,   we used the fact that  $\frac{|x|^2}{16t^2}$ can be upper bounded by $\frac{|x|}{4t^2}$ since $w(\cdot, t)$ is supported in $B_4$. \\
We then simplify $I_5$ by first writing $2 \w \Z \w = \Z \w^2$ to get 
\begin{align*}
I_5=	2\int \left\langle \frac{Ax}{4t}, \frac{x}{4t}\right\rangle\w \frac{\Z \w}{2t}t^{-\mu}\left(\frac{t \s'}{\s}\right)^{-1/2}=\int \frac{\langle Ax, x \rangle}{16t^2}\left( (\w^2)_t
+ \langle A \n \w^2, \frac{x}{2t}\rangle\right)t^{-\mu}\left(\frac{t \s'}{\s}\right)^{-1/2}.
\end{align*} 
 Integration by parts, we  obtain
\begin{align*}
	I_5&=-\int \frac{\langle Ax, x \rangle}{16}\w^2\left( t^{-\mu-2}\left(\frac{t \s'}{\s}\right)^{-1/2}\right)'
	-\int_{\{t=a\}} \frac{\langle Ax, x \rangle}{16t^2}\w^2 t^{-\mu}\left(\frac{t \s'}{\s}\right)^{-1/2} \\
	&-\int \frac{\D(\langle Ax, x \rangle A x)}{32t^3}\w^2 t^{-\mu}\left(\frac{t \s'}{\s}\right)^{-1/2}.
\end{align*}
First note that $\D(\langle Ax, x\rangle A x)= (n+2)|x|^2 + O(1)|x|^3.$ Consequently after simplification we get the following
\begin{align*}
I_5&=\frac{\mu +2}{16}\int \langle Ax, x \rangle\w^2 t^{-\mu-3}\left(\frac{t \s'}{\s}\right)^{-1/2}+\frac{1}{32} \int \langle Ax, x \rangle\w^2 t^{-\mu-2}\left(\frac{t \s'}{\s}\right)^{-3/2}\left(\frac{t \s'}{\s}\right)' \\
&-\int_{\{t=a\}} \frac{\langle Ax, x \rangle}{16t^2}\w^2 t^{-\mu}\left(\frac{t \s'}{\s}\right)^{-1/2} 
	-\frac{(n+2)}{32}\int \frac{|x|^2}{t^3}\w^2 t^{-\mu}\left(\frac{t \s'}{\s}\right)^{-1/2}\\&-O(1)\int \frac{|x|^3}{32t^3}\w^2 t^{-\mu}\left(\frac{t \s'}{\s}\right)^{-1/2}.
\end{align*}
After  re-writing $A=A-\mathbb I_n+\mathbb I_n$  and using \eqref{ode} we obtain
\begin{align*}
	I_5&=\frac{\mu +2}{16}\int \langle (A-\mathbb I_n)x, x \rangle \w^2 t^{-\mu-3}\left(\frac{t \s'}{\s}\right)^{-1/2}+\frac{\mu +2}{16}\int |x|^2\w^2 t^{-\mu-3}\left(\frac{t \s'}{\s}\right)^{-1/2}\\
	&-\frac{1}{32} \int \langle (A-\mathbb I_n)x, x \rangle\w^2 t^{-\mu-2}\left(\frac{t \s'}{\s}\right)^{-1/2}\frac{\theta(\lambda t)}{t} -\frac{1}{32} \int |x|^2\w^2 t^{-\mu-2}\left(\frac{t \s'}{\s}\right)^{-1/2}\frac{\theta(\lambda t)}{t}\\
	& -\int_{\{t=a\}} \frac{\langle (A-\mathbb I_n)x, x \rangle}{16t^2}\w^2 t^{-\mu}\left(\frac{t \s'}{\s}\right)^{-1/2} -\int_{\{t=a\}} \frac{|x|^2}{16t^2}\w^2 t^{-\mu}\left(\frac{t \s'}{\s}\right)^{-1/2} 
	-\frac{(n+2)}{32}\int \frac{|x|^2}{t^3}\w^2 t^{-\mu}\left(\frac{t \s'}{\s}\right)^{-1/2}\\
	&-O(1)\int \frac{|x|^3}{32t^3}\w^2 t^{-\mu}\left(\frac{t \s'}{\s}\right)^{-1/2}.
\end{align*}
Since $\mu =\frac{n}{2}-1,$ therefore $\frac{\mu +2}{16}=\frac{n+2}{32}.$ Hence we find
\begin{align}\label{i52}
	I_5&=\frac{\mu +2}{16}\int \langle (A-\mathbb I_n)x, x \rangle \w^2 t^{-\mu-3}\left(\frac{t \s'}{\s}\right)^{-1/2}\\
		&-\frac{1}{32} \int \langle (A-\mathbb I_n)x, x \rangle\w^2 t^{-\mu-2}\left(\frac{t \s'}{\s}\right)^{-1/2}\frac{\theta(\lambda t)}{t}  -\frac{1}{32} \int |x|^2\w^2 t^{-\mu-2}\left(\frac{t \s'}{\s}\right)^{-1/2}\frac{\theta(\lambda t)}{t}\notag \\
		&-\int_{\{t=a\}} \frac{\langle (A-\mathbb I_n)x, x \rangle}{16t^2}\w^2 t^{-\mu}\left(\frac{t \s'}{\s}\right)^{-1/2}  -\int_{\{t=a\}} \frac{|x|^2}{16t^2}\w^2 t^{-\mu}\left(\frac{t \s'}{\s}\right)^{-1/2} -O(1)\int \frac{|x|^3}{32t^3}\w^2 t^{-\mu}\left(\frac{t \s'}{\s}\right)^{-1/2}.\notag
	\end{align}
For the first term in right-hand side of the \eqref{i52}, we substitute the value of $\w$ in terms of $w$, use  $|A(x,0)- \mathbb I_n|= O(|x|)$  and $\left(\frac{t \s'}{\s}\right)^{-1/2}$ is bounded, to obtain
\begin{align}\label{i521}
	\left|\frac{\mu +2}{16}\int \langle (A-\mathbb I_n)x, x \rangle \w^2 t^{-\mu-3}\left(\frac{t \s'}{\s}\right)^{-1/2}\right| \le N\int \frac{|x|^3}{t^2}\s^{-2\A}w^2 G .
\end{align}
Now for the second term in right hand side of the \eqref{i52}, along with the facts used to bound first term, we also use the fact that $\theta$ is bounded if $\lambda t \le 1$ to find
\begin{align}\label{i522}
	\left|\frac{1}{32} \int \langle (A-\mathbb I_n)x, x \rangle\w^2 t^{-\mu-2}\left(\frac{t \s'}{\s}\right)^{-1/2}\frac{\theta(\lambda t)}{t}\right| \le N \int \frac{|x|^3}{t^2}\s^{-2\A}w^2 G.
\end{align}
Therefore using \eqref{i521} and \eqref{i522} in \eqref{i52} we get
\begin{align}\label{i53}
	I_5 &\ge -\frac{1}{32} \int |x|^2\w^2t^{-\mu-2} \left(\frac{t \s'}{\s}\right)^{-1/2}\frac{\theta(\lambda t)}{t}
 -N\int \frac{|x|^3}{t^2}\s^{-2\A}w^2G\\
 &	-\int_{\{t=a\}} \frac{\langle (A-\mathbb I_n)x, x \rangle}{16t^2}\w^2 t^{-\mu}\left(\frac{t \s'}{\s}\right)^{-1/2}	-\int_{\{t=a\}} \frac{|x|^2}{16t^2}\w^2 t^{-\mu}\left(\frac{t \s'}{\s}\right)^{-1/2}.\notag
\end{align}
  Now the second term of right-hand side of \eqref{i53} is estimated using Lemma \ref{logi}  and we thus  find
\begin{align}\label{i5ff}
		I_5 &\ge -\frac{1}{32} \int \frac{|x|^2}{t}\s^{-2\A}w^2 \left(\frac{t \s'}{\s}\right)^{-1/2}\frac{\theta(\lambda t)}{t}
	-\int_{\{t=a\}} \frac{\langle (A-\mathbb I_n)x, x \rangle}{16t}\s^{-2\A}w^2G \left(\frac{t \s'}{\s}\right)^{-1/2}\\
	 & -\int_{\{t=a\}} \frac{|x|^2}{16t}\s^{-2\A}w^2G\left(\frac{t \s'}{\s}\right)^{-1/2} -N\A N^{2\A}\lambda^{2\A+N}\int w^2- N \A \delta \int \s^{-2\A}\frac{\theta(\lambda t)}{t}w^2 G.\notag
\end{align}
We now simplify $I_6.$
\begin{align*}
	I_6&= 2 \int t^{-\mu}\D(A\n \w)\frac{ \Z \w}{2t}\left(\frac{t \s'}{\s}\right)^{-1/2}\\
	&=2 \int t^{-\mu}\D(A\n \w)\w_t\left(\frac{t \s'}{\s}\right)^{-1/2} + \int t^{-\mu-1}\D(A\n \w)\langle A \n \w, x\rangle \left(\frac{t \s'}{\s}\right)^{-1/2}\\
	&=I_6^1+I_6^2.
\end{align*}
We first look at $I_6^1.$ We use integrate by parts formula in $x$-variable to obtain
\begin{align*}
	I_6^1&=2 \int t^{-\mu}\D(A\n \w)\w_t\left(\frac{t \s'}{\s}\right)^{-1/2}\\
	&=-2 \int t^{-\mu}\langle A\n \w, \n \w_t\rangle \left(\frac{t \s'}{\s}\right)^{-1/2}\\
	&=- \int t^{-\mu}\left(\frac{t \s'}{\s}\right)^{-1/2}\frac{d}{dt}\langle A\n \w, \n \w\rangle.
\end{align*}
Subsequently,   by integrating by parts  in  the $t$-variable we find
\begin{align*}
		I_6^1&= \int \langle A\n \w, \n \w\rangle \left(t^{-\mu}\left(\frac{t \s'}{\s}\right)^{-1/2}\right)' +\int_{\{t=a\}} \langle A\n \w, \n \w\rangle t^{-\mu}\left(\frac{t \s'}{\s}\right)^{-1/2}.
\end{align*}
After simplifying and using \eqref{ode} we obtain
\begin{align}\label{i61}
		I_6^1&= -\mu \int \langle A\n \w, \n \w\rangle t^{-\mu-1}\left(\frac{t \s'}{\s}\right)^{-1/2} +\frac{1}{2}\int \langle A\n \w, \n \w\rangle t^{-\mu}\left(\frac{t \s'}{\s}\right)^{-1/2}\frac{\theta(\lambda t)}{t}\\
		&+\int_{\{t=a\}} \langle A\n \w, \n \w\rangle t^{-\mu}\left(\frac{t \s'}{\s}\right)^{-1/2}.\notag
	\end{align}
We will now focus our attention on $I_6^2.$ After using integration by parts formula in $x$-variable we get
\begin{align*}
	I_6^2&=-\int t^{-\mu-1}(A\n \w)\cdot \n(\langle A \n \w, x\rangle) \left(\frac{t \s'}{\s}\right)^{-1/2}\\
	&=-\int t^{-\mu-1} \sum_{i,j,k,l} a_{ij}\w_i\frac{\partial}{\partial x_j}(a_{kl}\w_k x_l) \left(\frac{t \s'}{\s}\right)^{-1/2}\\
	&=-\int t^{-\mu-1}\sum_{i,j,k,l} a_{ij}\w_i\partial_{x_j}a_{kl}\w_k x_l \left(\frac{t \s'}{\s}\right)^{-1/2}-\int t^{-\mu-1}\sum_{i,j,k,l} a_{ij}\w_i a_{kl}\w_{kj} x_l \left(\frac{t \s'}{\s}\right)^{-1/2}\\
	&-\int t^{-\mu-1}\sum_{i,j,k,l} a_{ij}\w_ia_{kl}\w_k \delta_{lj} \left(\frac{t \s'}{\s}\right)^{-1/2}.
\end{align*}
We next observe that \begin{align}
	\langle Ax, \n \langle A \n \w, \n \w \rangle\rangle&=\sum_{i,j,k,l}a_{kl}x_l \partial_{x_k}a_{ij}\w_i \w_j +2 \sum_{i,j,k,l}a_{kl}x_la_{ij}\w_i \w_{jk}\\
&	=\sum_{i,j,k,l}a_{kl}x_l \partial_{x_k}a_{ij}\w_i \w_j+2\sum_{i,j,k,l} a_{ij}\w_ia_{kl}\w_{kj} x_l.\notag
\end{align}
 Hence $I_6^2$
 becomes 
 \begin{align}\label{i621}
 	I_6^2&=-\int t^{-\mu-1} \sum_{i,j,k,l} a_{ij}\w_i\partial x_ja_{kl}\w_k x_l \left(\frac{t \s'}{\s}\right)^{-1/2} + O(1) \int |x| t^{-\mu-1} |\nabla \w|^2 \\
 	&-\frac{1}{2}\int t^{-\mu-1}\langle Ax, \n \langle A \n \w, \n \w \rangle\rangle\left(\frac{t \s'}{\s}\right)^{-1/2}
 	-\int t^{-\mu-1}|A\n \w|^2\left(\frac{t \s'}{\s}\right)^{-1/2}.\notag
\end{align}
We again apply integration by parts formula in $x$-variable to the third  term on the right-hand side of \eqref{i621} to obtain 
\begin{align}\label{i622}
	I_6^2&=\int t^{-\mu-1}\sum_{i,j,k,l} a_{ij}\w_i\partial_{x_j}a_{kl}\w_k x_l \left(\frac{t \s'}{\s}\right)^{-1/2} + O(1) \int |x| t^{-\mu-1} |\nabla \w|^2 \\
	&+\frac{1}{2}\int t^{-\mu-1}\D( Ax) \langle A \n \w, \n \w \rangle\left(\frac{t \s'}{\s}\right)^{-1/2}
	-\int t^{-\mu-1}|A\n \w|^2\left(\frac{t \s'}{\s}\right)^{-1/2}.\notag
\end{align}
Again using $\D(Ax)=n +O(1)|x|$ in the third term and by splitting $A$ as  $A=A-\mathbb I_n+\mathbb I_n$ in the last term of right-hand side of \eqref{i622} and also using $|A-\mathbb I_n|= O(|x|)$, we find
\begin{align}\label{i6a}
	I_6^2&= O(1) \int |x| t^{-\mu-1} |\nabla \w|^2
	+\left(\frac{n}{2}-1\right)\int t^{-\mu-1} \langle A \n \w, \n \w \rangle\left(\frac{t \s'}{\s}\right)^{-1/2}.
\end{align}

Hence using \eqref{i61} and \eqref{i6a}, $I_6$ becomes 
\begin{align}\label{i6r}
	I_6&=-\mu \int \langle A\n \w, \n \w\rangle t^{-\mu-1}\left(\frac{t \s'}{\s}\right)^{-1/2} +\frac{1}{2}\int \langle A\n \w, \n \w\rangle t^{-\mu}\left(\frac{t \s'}{\s}\right)^{-1/2}\frac{\theta(\lambda t) }{t}\\
	&+\int_{\{t=a\}} \langle A\n \w, \n \w\rangle t^{-\mu}\left(\frac{t \s'}{\s}\right)^{-1/2}
	+\frac{n-2}{2}\int t^{-\mu-1} \langle A \n \w, \n \w \rangle\left(\frac{t \s'}{\s}\right)^{-1/2}\notag\\
	&+O(1)\int |x|t^{-\mu-1} | \n \w|^2\left(\frac{t \s'}{\s}\right)^{-1/2}.\notag
\end{align}
Since $\mu= \frac{n}{2}-1$, the  following terms in the right hand side of \eqref{i6r}, i.e. 
\[
-\mu \int \langle A\n \w, \n \w\rangle t^{-\mu-1}\left(\frac{t \s'}{\s}\right)^{-1/2}\]
and
\[
\frac{n-2}{2}\int t^{-\mu-1} \langle A \n \w, \n \w \rangle\left(\frac{t \s'}{\s}\right)^{-1/2}\]
cancel each other.
 Consequently we get
\begin{align}\label{i6f}
	I_6&= \frac{1}{2}\int \langle A\n \w, \n \w\rangle t^{-\mu}\left(\frac{t \s'}{\s}\right)^{-1/2}\frac{\theta (\lambda t)}{t}
	+\int_{\{t=a\}} \langle A\n \w, \n \w\rangle t^{-\mu}\left(\frac{t \s'}{\s}\right)^{-1/2}\\
&+O(1)\int |x|t^{-\mu-1} | \n \w|^2\left(\frac{t \s'}{\s}\right)^{-1/2}.\notag
	\end{align}
We now recall $ \n \w = \s^{-\A}e^{-|x|^2/8t}\n w -\frac{x}{4t}\s^{-\A}e^{-|x|^2/8t}w.$ This by AM-GM inequality implies  $$ |\n \w|^2 \le 2\s^{-2\A}e^{-|x|^2/4t}|\n w|^2 +\frac{|x|^2}{8t^2}\s^{-2\A}e^{-|x|^2/4t}w^2.$$
From Lemma \ref{sig}, we have $t/N\le \s\le t$ and $\left(\frac{t \s'}{\s}\right)^{-1/2}$ is bounded. Therefore last term in \eqref{i6f} can be estimated as
\begin{align}\label{l1}
	O(1)\int |x|t^{-\mu-1} | \n \w|^2\left(\frac{t \s'}{\s}\right)^{-1/2} \le O(1)\int \s^{1-2\A}\frac{|x|}{t}|\n w|^2 G + O(1)\int \s^{-2\A}\frac{|x|^3}{t^2}w^2G.
\end{align}
We now use Lemma \ref{logi} to obtain 
\begin{align}\label{l2}
	O(1)\int |x|t^{-\mu-1} | \n \w|^2\left(\frac{t \s'}{\s}\right)^{-1/2} \ge -NN^{2\A}\lambda^{2\A+N}\int t|\n w|^2- N \delta \int \s^{1-2\A}\frac{\theta(\lambda t)}{t}|\n w|^2 G\\
	-\A NN^{2\A}\lambda^{2\A+N}\int w^2- \A N \delta \int \s^{-2\A}\frac{\theta(\lambda t)}{t}w^2 G.\notag
\end{align}
Hence using \eqref{l2} in \eqref{i6f} we  deduce the following inequality
\begin{align}\label{it}
	I_6 &\ge \frac{1}{2}\int \langle A\n \w, \n \w\rangle t^{-\mu}\left(\frac{t \s'}{\s}\right)^{-1/2}\frac{\theta (\lambda t)}{t}
	+\int_{\{t=a\}} \langle A\n \w, \n \w\rangle t^{-\mu}\left(\frac{t \s'}{\s}\right)^{-1/2}\\
	&-NN^{2\A}\lambda^{2\A+N}\int t|\n w|^2- N \delta \int \s^{1-2\A}\frac{\theta(\lambda t)}{t}|\n w|^2 G\notag\\
	&-\A NN^{2\A}\lambda^{2\A+N}\int w^2- \A N \delta \int \s^{-2\A}\frac{\theta(\lambda t)}{t}w^2 G.\notag
\end{align}
We  now estimate the boundary term in \eqref{it}, i,e. the integral $\int_{\{t=a\}} \langle A\n \w, \n \w\rangle t^{-\mu}\left(\frac{t \s'}{\s}\right)^{-1/2} $. By recalling the definition of $\w$ in term of $w,$ 
 we have $$\langle A\n \w, \n \w\rangle = \s^{-2\A}e^{-|x|^2/4t}\left(\langle A \n w, \n w\rangle -2\langle w \frac{Ax}{4t}, \n w\rangle + \frac{w^2}{16t^2}\langle Ax,x \rangle\right).$$
Furthermore using $\mu =\frac{n}{2}-1$ we get 
\begin{align}\label{i6b}
	\int_{\{t=a\}} \langle A\n \w, \n \w\rangle t^{-\mu}\left(\frac{t \s'}{\s}\right)^{-1/2}&=\int_{\{t=a\}}\s^{-2\A}e^{-\frac{|x|^2}{4t}}\left(\langle A \n w, \n w\rangle -\langle w \frac{Ax}{2t}, \n w\rangle + \frac{w^2}{16t^2}\langle Ax,x \rangle\right) t^{-\frac{n}{2}+1}\left(\frac{t \s'}{\s}\right)^{-\frac12}\notag\\
	&=\int_{\{t=a\}}t\s^{-2\A}\langle A \n w, \n w\rangle G\left(\frac{t \s'}{\s}\right)^{-\frac12}-\frac{1}{2}\int_{\{t=a\}}\s^{-2\A}\langle w Ax, \n w\rangle G\left(\frac{t \s'}{\s}\right)^{-\frac12}\\
	&+\frac{1}{16}\int_{\{t=a\}}\s^{-2\A}\frac{w^2}{t}\langle Ax,x \rangle G\left(\frac{t \s'}{\s}\right)^{-\frac12}.\notag
\end{align}
We now focus our attention on the second term in the right-hand side of \eqref{i6b}. We first write $w \n w = \n w^2/2$ and then use integration by parts formula to obtain 
\begin{align*}
	&-\frac{1}{2}\int_{\{t=a\}}\s^{-2\A}\langle w Ax, \n w\rangle G\left(\frac{t \s'}{\s}\right)^{-\frac12}=-\frac{1}{4}\int_{\{t=a\}}\s^{-2\A}\langle  Ax, \n w^2\rangle G\left(\frac{t \s'}{\s}\right)^{-\frac12}\\
	&=\frac{1}{4}\int_{\{t=a\}}\s^{-2\A}w^2 \D(Ax)G\left(\frac{t \s'}{\s}\right)^{-\frac12} +\frac{1}{4}\int_{\{t=a\}}\s^{-2\A}\langle  Ax, \n G\rangle w^2\left(\frac{t \s'}{\s}\right)^{-\frac12}.
\end{align*}
After using $\D(Ax) =n +O(|x|)$ and $\n G= -x/2t G$, we find
\begin{align}\label{i6b2}
	&-\frac{1}{2}\int_{\{t=a\}}\s^{-2\A}\langle w Ax, \n w\rangle G\left(\frac{t \s'}{\s}\right)^{-\frac12}\\
	&=\frac{n}{4}\int_{\{t=a\}}\s^{-2\A}w^2 G\left(\frac{t \s'}{\s}\right)^{-\frac12} +O(1)\int_{\{t=a\}}|x|\s^{-2\A}w^2 G\left(\frac{t \s'}{\s}\right)^{-\frac12}-\frac{1}{4}\int_{\{t=a\}}\s^{-2\A}\langle  Ax,\frac{x}{2t} \rangle w^2 G\left(\frac{t \s'}{\s}\right)^{-\frac12}.\notag
\end{align}
Hence using \eqref{i6b2} in \eqref{i6b}, the boundary term in $I_6$ becomes
\begin{align}\label{bgw}
		\int_{\{t=a\}} \langle A\n \w, \n \w\rangle t^{-\mu}\left(\frac{t \s'}{\s}\right)^{-1/2} &= \int_{\{t=a\}}t\s^{-2\A}\langle A \n w, \n w\rangle G\left(\frac{t \s'}{\s}\right)^{-\frac12}
		+\frac{n}{4}\int_{\{t=a\}}\s^{-2\A}w^2 G\left(\frac{t \s'}{\s}\right)^{-\frac12}\\ &+O(1)\int_{\{t=a\}}|x|\s^{-2\A}w^2 G\left(\frac{t \s'}{\s}\right)^{-\frac12}-\frac{1}{4}\int_{\{t=a\}}\s^{-2\A}\langle  Ax,\frac{x}{2t} \rangle w^2 G\left(\frac{t \s'}{\s}\right)^{-\frac12}\notag\\
		&+\frac{1}{16}\int_{\{t=a\}}\s^{-2\A}\frac{w^2}{t}\langle Ax,x \rangle G\left(\frac{t \s'}{\s}\right)^{-\frac12}.\notag
\end{align}
 Using $|x|<4$ and also by writing $A=A-\mathbb I_n +\mathbb I_n$  we get
\begin{align}\label{6bf}
	\int_{\{t=a\}} \langle A\n \w, \n \w\rangle t^{-\mu}\left(\frac{t \s'}{\s}\right)^{-1/2}
	& \ge \int_{\{t=a\}}t\s^{-2\A}\langle A \n w, \n w\rangle G\left(\frac{t \s'}{\s}\right)^{-\frac12}
-N\int_{\{t=a\}}\s^{-2\A}w^2 G\left(\frac{t \s'}{\s}\right)^{-\frac12}\\
	&-\frac{1}{16}\int_{\{t=a\}}\s^{-2\A}\langle  (A-\mathbb I_n)x,x \rangle \frac{w^2}{t
	} G\left(\frac{t \s'}{\s}\right)^{-\frac12}-\frac{1}{16}\int_{\{t=a\}}\s^{-2\A}\frac{|x|^2}{t} \w^2
 G\left(\frac{t \s'}{\s}\right)^{-\frac12}.\notag
\end{align}
 Thus using \eqref{6bf} in \eqref{it} we find
\begin{align}\label{i6ff1}
		I_6& \ge  \frac{1}{2}\int  \langle A\n \w,\n \w \rangle  t^{-\mu}\frac{\theta (\lambda t)}{t}\left(\frac{t \s'}{\s}\right)^{-\frac12}
		-NN^{2\A}\lambda^{2\A+N}\int t|\n w|^2- N \delta \int \s^{1-2\A}\frac{\theta(\lambda t)}{t}|\n w|^2 G\\
		&-\A NN^{2\A}\lambda^{2\A+N}\int w^2- \A N \delta \int \s^{-2\A}\frac{\theta(\lambda t)}{t}w^2 G +\int_{\{t=a\}}t\s^{-2\A}\langle A \n w, \n w\rangle G\left(\frac{t \s'}{\s}\right)^{-\frac12}\notag\\
		&-N\int_{\{t=a\}}\s^{-2\A}w^2 G\left(\frac{t \s'}{\s}\right)^{-\frac12}
		-\frac{1}{16}\left[\int_{\{t=a\}}\s^{-2\A}\langle  (A-\mathbb I_n)x,x \rangle \frac{w^2}{t
		} G\left(\frac{t \s'}{\s}\right)^{-\frac12}+\int_{\{t=a\}}\s^{-2\A}\frac{|x|^2}{t} w^2
		G\left(\frac{t \s'}{\s}\right)^{-\frac12}\right].\notag
\end{align}
By similar  computations as in  \eqref{bgw}  we obtain
\begin{align}
		\int \langle A\n \w, \n \w\rangle t^{-\mu}\left(\frac{t \s'}{\s}\right)^{-1/2}\frac{\theta(\lambda t)}{t} &= \int t\s^{-2\A}\langle A \n w, \n w\rangle G\left(\frac{t \s'}{\s}\right)^{-\frac12}\frac{\theta(\lambda t)}{t}
	+\frac{n}{4}\int\s^{-2\A}w^2 G\left(\frac{t \s'}{\s}\right)^{-\frac12}\frac{\theta(\lambda t)}{t}\\ &+O(1)\int|x|\s^{-2\A}w^2 G\left(\frac{t \s'}{\s}\right)^{-\frac12}\frac{\theta(\lambda t)}{t}\notag\\
	&- \frac{1}{16}\int\s^{-2\A}\frac{w^2}{t}\langle Ax,x \rangle G\left(\frac{t \s'}{\s}\right)^{-\frac12}\frac{\theta(\lambda t)}{t}.\notag
\end{align}
Now writing $A=A-\mathbb I_n+\mathbb I_n$ and using $|A-\mathbb I_n| = O(|x|)$ we get
\begin{align}\label{abc}
	\int \langle A\n \w, \n \w\rangle t^{-\mu}\left(\frac{t \s'}{\s}\right)^{-1/2}\frac{\theta(\lambda t)}{t} &\ge \int t\s^{-2\A}\langle A \n w, \n w\rangle G\left(\frac{t \s'}{\s}\right)^{-\frac12}\frac{\theta(\lambda t)}{t}
	+\frac{n}{4}\int\s^{-2\A}w^2 G\left(\frac{t \s'}{\s}\right)^{-\frac12}\frac{\theta(\lambda t)}{t}\\ &+O(1)\int|x|\s^{-2\A}w^2 G\left(\frac{t \s'}{\s}\right)^{-\frac12}\frac{\theta(\lambda t)}{t}-\frac{1}{16}\int\s^{-2\A}\frac{|x|^2}{t} G\left(\frac{t \s'}{\s}\right)^{-\frac12}\frac{\theta(\lambda t)}{t}\notag\\
	&- N\int\s^{-2\A}\frac{|x|^3}{t} G\left(\frac{t \s'}{\s}\right)^{-\frac12}\frac{\theta(\lambda t)}{t}.\notag
\end{align}
Now in the third term and the last term of \eqref{abc} above, we use $\theta$ is bounded as $\lambda t \le 1$ and Lemma \ref{logi} to obtain
\begin{align}
	\int \langle A\n \w, \n \w\rangle t^{-\mu}\left(\frac{t \s'}{\s}\right)^{-1/2}\frac{\theta(\lambda t)}{t} &\geq \int t\s^{-2\A}|\n w|^2 G\left(\frac{t \s'}{\s}\right)^{-\frac12}\frac{\theta(\lambda t)}{t}
	+\frac{n}{4}\int\s^{-2\A}w^2 G\left(\frac{t \s'}{\s}\right)^{-\frac12}\frac{\theta(\lambda t)}{t}\\ &- \A NN^{2\A}\lambda^{2\A+N}\int w^2-  N \A \delta \int \s^{-2\A}\frac{\theta(\lambda t)}{t}w^2 G\notag\\
	&-\frac{1}{16}\int\s^{-2\A}\frac{|x|^2}{t} G\left(\frac{t \s'}{\s}\right)^{-\frac12}\frac{\theta(\lambda t)}{t}.\notag
\end{align}
We thus have
\begin{align}\label{i6ff}
		I_6& \ge  \frac{1}{N}\int \s^{1-2\A}  |\n w|^2  G \frac{\theta (\lambda t)}{t}\left(\frac{t \s'}{\s}\right)^{-\frac12}-\frac{1}{32}\int\s^{-2\A}\frac{|x|^2}{t} G\left(\frac{t \s'}{\s}\right)^{-\frac12}\frac{\theta(\lambda t)}{t}+\frac{n}{8}\int\s^{-2\A}w^2 G\left(\frac{t \s'}{\s}\right)^{-\frac12}\frac{\theta(\lambda t)}{t}\\
	&-NN^{2\A}\lambda^{2\A+N}\int t|\n w|^2- N \delta \int \s^{1-2\A}\frac{\theta(\lambda t)}{t}|\n w|^2 G \notag\\
	&-\A NN^{2\A}\lambda^{2\A+N}\int w^2- \A N \delta \int \s^{-2\A}\frac{\theta(\lambda t)}{t}w^2 G +\int_{\{t=a\}}t\s^{-2\A}\langle A \n w, \n w\rangle G\left(\frac{t \s'}{\s}\right)^{-\frac12}\notag\\
	&-N\int_{\{t=a\}}\s^{-2\A}w^2 G\left(\frac{t \s'}{\s}\right)^{-\frac12}
	-\frac{1}{16}\left[\int_{\{t=a\}}\s^{-2\A}\langle  (A-\mathbb I_n)x,x \rangle \frac{w^2}{t
	} G\left(\frac{t \s'}{\s}\right)^{-\frac12}+\int_{\{t=a\}}\s^{-2\A}\frac{|x|^2}{t} w^2
	G\left(\frac{t \s'}{\s}\right)^{-\frac12}\right].\notag
\end{align}
We now focus our attention on $I_7.$ Again using $2 \w \Z \w = \Z \w^2,$ we get
\begin{align*}
	I_7&= \int V \frac{\Z \w^2}{2t}t^{-\mu}\left(\frac{t \s'}{\s}\right)^{-1/2}=\int V (\partial_t \w^2 + \langle A \n \w^2, \frac{x}{2t}\rangle)t^{-\mu}\left(\frac{t \s'}{\s}\right)^{-1/2}.
\end{align*}
After using integratation by parts formula we obtain 
\begin{align*}
	I_7&= -\int V \w^2 \left(t^{-\mu}\left(\frac{t \s'}{\s}\right)^{-1/2}\right)'-\int_{\{t=a\}}  V \w^2 t^{-\mu}\left(\frac{t \s'}{\s}\right)^{-1/2}
	-\frac{1}{2}\int \D(V Ax)\w^2t^{-\mu-1}\left(\frac{t \s'}{\s}\right)^{-1/2}.
\end{align*}
It is easy to see that after some simplification and using \eqref{ode} we get
\begin{align}
	I_7&= \mu\int V \w^2 t^{-\mu-1}\left(\frac{t \s'}{\s}\right)^{-1/2}-\frac{1}{2}\int V \w^2 t^{-\mu}\left(\frac{t \s'}{\s}\right)^{-1/2}\frac{\theta (\lambda t)}{t}
	-\int_{\{t=a\}}  V \w^2 t^{-\mu}\left(\frac{t \s'}{\s}\right)^{-1/2}\\
	&-\frac{1}{2}\int \langle \n V, Ax\rangle \w^2 t^{-\mu-1}\left(\frac{t \s'}{\s}\right)^{-1/2} -\frac{1}{2}\int   V \D(Ax) \w^2 t^{-\mu-1}\left(\frac{t \s'}{\s}\right)^{-1/2}.\notag
\end{align}
We now use $\D(Ax)=n +O(|x|)$ and rearrange terms to obtain
\begin{align}
	I_7&= \left(\mu-\frac{n}{2}\right)\int V \w^2 t^{-\mu-1}\left(\frac{t \s'}{\s}\right)^{-1/2}-\frac{1}{2}\int V \w^2 t^{-\mu}\left(\frac{t \s'}{\s}\right)^{-1/2}\frac{\theta (\lambda t)}{t}
	-\int_{\{t=a\}}  V \w^2 t^{-\mu}\left(\frac{t \s'}{\s}\right)^{-1/2}\\
	&-\frac{1}{2}\int \langle \n V, Ax\rangle \w^2 t^{-\mu-1}\left(\frac{t \s'}{\s}\right)^{-1/2} +O(1)\int |x| V  \w^2 t^{-\mu-1}\left(\frac{t \s'}{\s}\right)^{-1/2}.\notag
\end{align}
Since $\lambda t \le 1,$ using the boundedness of  $\theta$ and also that $Ax=O( |x|)$, we deduce the following inequality
\begin{align}
	I_7&\ge \left(\mu-\frac{n}{2}\right)\int |V| \w^2 t^{-\mu-1}\left(\frac{t \s'}{\s}\right)^{-1/2}-\frac{N}{2}\int |V| \w^2 t^{-\mu-1}\left(\frac{t \s'}{\s}\right)^{-1/2}
	-\int_{\{t=a\}} t V \w^2 t^{-\mu-1}\left(\frac{t \s'}{\s}\right)^{-1/2}\\
	&-\frac{N}{2}\int | \n V||x| \w^2 t^{-\mu-1}\left(\frac{t \s'}{\s}\right)^{-1/2} -N\int |x| |V | \w^2 t^{-\mu-1}\left(\frac{t \s'}{\s}\right)^{-1/2}.\notag
\end{align}
 Now using $|x| \leq B_4$ and by writing $\w$ in terms of $w$ we get 
\begin{align}\label{i7ff}
	I_7 \ge -N(||V(\cdot, 0)||_{\infty}+ ||\n V(\cdot, 0)||_{\infty}) \int \s^{-2\A}w^2 G \left(\frac{t \s'}{\s}\right)^{-1/2} -\int_{\{t=a\}}  t \s^{-2\A}V w^2 G \left(\frac{t \s'}{\s}\right)^{-1/2}.
\end{align}
We now substitute the bounds for $I_2$, $I_3$,...,$I_7$  as in  \eqref{i2ff}, \eqref{i3ff}, \eqref{i4ff}, \eqref{i5ff}, \eqref{i6ff}, \eqref{i7ff}  respectively in \eqref{amgm} to obtain
\begin{align}
		&\int \s^{-2\A}e^{\frac{-|x|^2}{4t}}\left( w_t + \D(A \n w) + V w\right)^2t^{-\mu}\left(\frac{t \s'}{\s}\right)^{-1/2}\\
	&\ge \int \left(\frac{\Z\w}{2t}\right)^2t^{-\mu}\left(\frac{t \s'}{\s}\right)^{-1/2} -\frac{1}{4}\int \left(\frac{\Z \w}{2t}\right)^2 t^{-\mu}\left(\frac{t \s'}{\s}\right)^{-1/2}\notag\\
	&+\left(\frac{\A}{2}-\frac{n}{8}+\frac{n}{8N}\right)\int \s^{-2\A}\left(\frac{t \s'}{\s}\right)^{1/2}\frac{\theta (\lambda t) }{t}w^2 G+	\left(\frac{1}{16}-\frac{1}{32}-\frac{1}{32}\right)\int \s^{-2\A}\frac{|x|^2}{t} \left(\frac{t \s'}{\s}\right)^{-\frac{1}{2}}\frac{\theta(\lambda t)}{t}w^2 G\notag\\
	&-N||V(\cdot, 0)||_1\int \s^{-2\A}w^2G\left(\frac{t \s'}{\s}\right)^{-\frac{1}{2}}-N(1+\A) N^{2\A}\lambda^{2\A+N} \int w^2 -N(1+\A)\delta \int \s^{-2\A}\frac{\theta(\lambda t)}{t}w^2G\notag\\
	&+ \frac{1}{N}\int \s^{1-2\A}|\n w|^2 G\frac{\theta (\lambda t)}{t}
	-NN^{2\A}\lambda^{2\A+N}\int t|\n w|^2- N \delta \int \s^{1-2\A}\frac{\theta(\lambda t)}{t}|\n w|^2 G\notag\\
&+\left(\frac{1}{8}-\frac{1}{16}-\frac{1}{16}\right)\int_{\{t=a\}}\s^{-2\A}\frac{|x|^2}{t} w^2
G\left(\frac{t \s'}{\s}\right)^{-\frac12}-2\int_{\{t=a\}} \s^{-2\A}\frac{\langle (A-\mathbb I_n)x, x \rangle}{16t}w^2G \left(\frac{t \s'}{\s}\right)^{-1/2}\notag\\
	& +\int_{\{t=a\}}t\s^{-2\A}\langle A \n w, \n w\rangle G\left(\frac{t \s'}{\s}\right)^{-\frac12}-(\A+n/4+N)\int_{\{t=a\}}\s^{-2\A}w^2 G\left(\frac{t \s'}{\s}\right)^{-\frac12}\notag\\
	&-||V(\cdot, 0)||_1\int_{\{t=a\}}t\s^{-2\A}w^2 G\left(\frac{t \s'}{\s}\right)^{-\frac12}.\notag
\end{align}
 We use $1/N\le \left(\frac{t \s'}{\s}\right)^{-\frac12}\le N,$ $\langle A \n w, \n w\rangle \sim  |\n w|^2$ and subsequently we obtain
\begin{align}\label{s11f}
&N\int \s^{-2\A}e^{\frac{-|x|^2}{4t}}\left( w_t + \D(A \n w) + V w\right)^2t^{-\mu}\\
&\ge \frac{3}{4}\int \left(\frac{\Z\w}{2t}\right)^2t^{-\mu}\left(\frac{t \s'}{\s}\right)^{-1/2}+\left(\frac{\A}{2N}-\frac{nN}{8}+\frac{n}{8N^2}-N(1+\A)\delta\right)\int \s^{-2\A}\frac{\theta (\lambda t) }{t}w^2 G-N^2||V(\cdot, 0)||_1\int \s^{-2\A}w^2 G\notag\\
&-N(1+\A) N^{2\A}\lambda^{2\A+N} \int w^2
+ \left(\frac{1}{N}-N\delta\right)\int \s^{1-2\A}|\n w|^2 G\frac{\theta (\lambda t)}{t}	-NN^{2\A}\lambda^{2\A+N}\int t|\n w|^2\notag\\
&-2\int_{\{t=a\}} \s^{-2\A}\frac{\langle (A-\mathbb I_n)x, x \rangle}{16t}w^2G  +\frac{1}{N}\int_{\{t=a\}}t\s^{-2\A}|\n w|^2 G\notag\\
&-N(\A+n/4+N-a||V(\cdot, 0)||_1)\int_{\{t=a\}}\s^{-2\A}w^2 G.\notag
\end{align}
We now estimate $2\int_{\{t=a\}} \s^{-2\A}\frac{\langle (A-\mathbb I_n)x, x \rangle}{16t}w^2G .$ To do this,  we first observe that since $\sigma(t) \sim t$,  therefore for $|x|>\delta,$ it is easy to see that $$\frac{|x|^3}{a}G(x,a)\s^{-2\A}(a) \le N^{2\A}\lambda^{2\A+N}.$$
Using this along with the fact that $|A(x, 0) - \mathbb I_n| = O(|x|)$, we find \begin{align}
	&2\left|\int_{\{t=a\}} \s^{-2\A}\frac{\langle (A-\mathbb I_n)x, x \rangle}{16t}w^2G\right| \le N \int \s^{-2\A}(a)\frac{|x|^3}{a}w^2(x,a)G(x,a) dx\\
	&= N \int_{B_{\delta}} \s^{-2\A}(a)\frac{|x|^3}{a}w^2(x,a)G(x,a) dx+ N \int_{\{|x|>\delta\}} \s^{-2\A}(a)\frac{|x|^3}{a}w^2(x,a)G(x,a) dx\notag\\
	&\le N\delta \int \s^{-2\A}(a)\frac{|x|^2}{a}w^2(x,a)G(x,a) dx+ N^{2\A}\lambda^{2\A+N} \int_{\{|x|>\delta\}} w^2(x,a) dx\notag\\
	&\le 2N^2\delta \int a\s^{-2\A}(a)|\n w(x,a)|^2G(x,a)+N^2\delta \int \s^{-2\A}(a)w^2(x,a)G(x,a) dx  + N^{2\A+1}\lambda^{2\A+N} \int w^2(x,a) dx,\notag
\end{align}
where in the last inequality we have used Lemma \ref{logb}. Hence \eqref{s11f} becomes
\begin{align}
	&N\int \s^{-2\A}e^{\frac{-|x|^2}{4t}}\left( w_t + \D(A \n w) + V w\right)^2t^{-\mu}\\
	&\ge \frac{3}{4}\int \left(\frac{\Z\w}{2t}\right)^2t^{-\mu}\left(\frac{t \s'}{\s}\right)^{-1/2}+\left(\frac{\A}{2N}-\frac{nN}{8}+\frac{n}{8N^2}-N(1+\A)\delta\right)\int \s^{-2\A}\frac{\theta (\lambda t) }{t}w^2 G-N^2||V(\cdot, 0)||_1\int \s^{-2\A}w^2 G\notag\\
	&-N(1+\A) N^{2\A}\lambda^{2\A+N} \underset{t\ge a}{\text{sup}}\int w^2(x,t)dx
	+ \left(\frac{1}{N}-N\delta\right)\int \s^{1-2\A}|\n w|^2 G\frac{\theta (\lambda t)}{t}	-NN^{2\A}\lambda^{2\A+N}\int t|\n w|^2\notag\\
	& +\left(\frac{1}{N}-N^2\delta\right)\int_{\{t=a\}}t\s^{-2\A}|\n w|^2 G -N(\A+n/4+N-a||V(\cdot, 0)||_1-N\delta)\int_{\{t=a\}}\s^{-2\A}w^2 G.\notag
\end{align}
Now we choose $\delta>0$ small enough such that $$\frac{\A}{2N}-\frac{nN}{8}+\frac{n}{8N^2}-N(1+\A)\delta \ge \frac{\A}{4N},$$ $$\frac{1}{N}-N\delta \ge \frac{1}{2N}$$ and $$\frac{1}{N}-N^2\delta \ge \frac{1}{2N}.$$  At this point,  we make the preliminary  yet crucial observation that  for $t \leq \frac{1}{2\lambda}$\begin{equation}\label{cr1}\frac{\theta (\lambda t) }{t} \geq \lambda^{1/2}t^{-1/2} (\log 2)^{3/2}  \sim \lambda.\end{equation} Hence we obtain
\begin{align}
	&\int \s^{-2\A}e^{\frac{-|x|^2}{4t}}\left( w_t + \D(A \n w) + V w\right)^2t^{-\mu}\left(\frac{t \s'}{\s}\right)^{-1/2}\ge \frac{\A\lambda}{4N}\int \s^{-2\A}w^2 G-N^2||V(\cdot, 0)||_1\int \s^{-2\A}w^2 G\notag\\
	&-N(1+\A) N^{2\A}\lambda^{2\A+N} \underset{t\ge a}{\text{sup}} \int w^2
	+ \frac{\lambda}{2N}\int \s^{1-2\A}|\n w|^2\frac{\theta(\lambda t)}{t} G	-NN^{2\A}\lambda^{2\A+N}\int t|\n w|^2\notag\\
	& +\frac{1}{2N}\int_{\{t=a\}}t\s^{-2\A}|\n w|^2 G -N^{2\A}\lambda^{2\A+N} \int_{\{t=a\}} w^2 -N(\A+n/4+N-a||V(\cdot, 0)||_1-N\delta)\int_{\{t=a\}}\s^{-2\A}w^2 G.\notag
\end{align}
If we now choose $\A^2 \ge 8N^3 \delta^2||V(\cdot, 0)||_1$ then $$\frac{\A\lambda}{4N}\int \s^{-2\A}w^2 G-N^2||V(\cdot, 0)||_1\int \s^{-2\A}w^2 G \ge \frac{\A\lambda}{8N}\int \s^{-2\A}w^2 G.$$
Moreover since  $a\A<1$ we find $$a||V(\cdot, 0)||_1 \le \frac{a\A^2}{ 8N^3 \delta^2}\le \frac{\A}{ 8N^3 \delta^2}.$$
 Consequently  for sufficiently large $N$, $\A\ge N||V(\cdot, 0)||^{1/2}_1+N$ and $0<a<1/\lambda$, we get
\begin{align}\label{f00}
	&N\int \s^{-2\A}e^{\frac{-|x|^2}{4t}}\left( w_t + \D(A \n w) + V w\right)^2t^{-\mu}\\
	&\ge \frac{\A^2}{N}\int \s^{-2\A}w^2 
	-N^{N\A}\lambda^{2\A}\underset{t \ge a}{\text{sup}} \int w^2(x,t)dx
	+ \frac{1}{N}\int \s^{1-2\A}|\n w|^2\frac{\theta(\lambda t)}{t} G	-N^{2\A}\A^{2\A}\int t|\n w|^2\notag\\
	& +\frac{1}{N}\int_{\{t=a\}}t\s^{-2\A}|\n w|^2 G -N^{2\A}\lambda^{2\A+N} \int_{\{t=a\}} w^2 -N\A\int_{\{t=a\}}\s^{-2\A}w^2 G.\notag
\end{align}
The estimate \eqref{step1} thus follows.

\textbf{Step 2}: We now show that there exists a universal constant $N$ and $\delta$ such that for all $\A \ge N(1+||V(\cdot, 0)||_1^{1/2})$ the following inequality holds
\begin{align}\label{cari}
	&	\A^2 \int \s_a^{-2\A}w^2G_a +\A \int \s_a^{1-2\A}|\n w|^2G_a\\
	& \le N \int \s_a^{1-2\A}(\D(A(x,t) \n w)+w_t+V(x,0)w)^2 G_a +N^{2\A}\A^{2\A}\underset{t \ge 0}{\operatorname{sup}}\int w^2
	+t|\n w|^2 dx\notag\\
	&+\s(a)^{-2\A}\left(-\frac{a}{N}\int |\n w (x,0)|^2G(x,a)dx + N \A \int w^2(x,0)G(x,a)dx\right).\notag
\end{align}
for all $0<a\le \frac{1}{4\lambda}$ and $w \in C_{0}^{\infty}(B_4 \times [0,\frac{1}{2\lambda}))$, where $\lambda=\A/\delta^2.$\\
Note that the difference from Step 1 is that we replace $A(x,0)$ by $A(x,t).$ This essentially follows from \eqref{step1} using similar ideas in \cite{EF}. We nevertheless provide all the details for the sake of completeness in order to track the precise dependence of the parameter $\alpha$.

\medskip 

\textbf{Proof of Step 2}:
  We write $$\D(A(x,0) \n w)+w_t+Vw=\D((A(x,0)-A(x,t))\n w)+\D(A(x,t) \n w)+w_t+Vw.$$ This gives
\begin{align}\label{br}
	N \int \s^{1-2\A}(\D(A(x,0) \n w)+w_t+Vw)^2G_a &\le 2	N \int \s_a^{1-2\A}(\D((A(x,0)-A(x,t)) \n w)^2 G_a\\
	& +2 	N \int \s_a^{1-2\A}(\D(A(x,t) \n w)+w_t+Vw)^2G_a.\notag
\end{align}
This suggests we have to estimate
\begin{align}
	N \int \s_a^{1-2\A}(\D((A(x,0)-A(x,t)) \n w)^2 G_a.
\end{align}
To simplify this term note that  $$\D((A(x,0)-A(x,t)) \n w)=\sum \partial_j(a_{ij}(x,0)-a_{ij}(x,t))w_i+(a_{ij}(x,0)-a_{ij}(x,t))w_{ij}.$$
Also, we have $|a_{ij}(x,0)-a_{ij}(x,t)| \le M \sqrt{t}\le M \sqrt{t+a}\le MN\s_a(t).$ Therefore we have
\begin{align}\label{dis}
	N \int \s_a^{1-2\A}(\D((A(x,0)-A(x,t)) \n w)^2 G_a \le 	NM^2 \int \s_a^{1-2\A} |\n w|^2 G_a + N^2 M^2 \int \s_a^{2-2\A}|D^2w|^2G_a.
\end{align}

Thus we will be done if we could estimate $\int \s_a^{2-2\A}|D^2w|^2G_a.$ We now proceed to estimate  $\int \s_a^{2-2\A}|D^2w|^2G_a.$  In the ensuing computations below, we let $A(\cdot)=A(\cdot,0)$ and $V=V(\cdot, 0)$.
First we calculate for $k=1,2,...,n$   
\begin{align}\label{ca}
	\partial_t w_k^2 + \D(A \n (w_k^2))&=2w_k \partial_t w_k + 2 \D(A w_k \n w_k)\\
	&=2w_k \partial_t w_k + 2 \D(A \n w_k)w_k + 2 \langle A \n w_k, \n w_k\rangle.\notag
\end{align}
We now define $H(w):=w_t+ \D(A(x,0) \n w)$. We let $\partial_k A$ be the matrix with entries $(\partial_k A)_{ij}=\partial_{k}a_{ij}.$ Now we calculate
 \begin{align*}
&\partial_k(w_k(\partial_t w + \D(A \n w)+Vw))\\
&= w_{kk}(H(w)+Vw) + w_k \partial_t w_k + w_k \D((\partial_k A)\n w) + w_k \D( A \n w_k)+w_kV_kw+Vw_k^2\\
&= w_{kk}(H(w)+Vw) + w_k \partial_t w_k  +\D(w_k(\partial_k A)\n w) -\langle \n w_k,(\partial_k A)\n w \rangle  + w_k \D( A \n w_k)+w_k V_kw + Vw_k^2,
\end{align*}
where we have used $ w_k \D((\partial_k A)\n w)=\D(w_k(\partial_k A)\n w) -\langle \n w_k,(\partial_k A)\n w \rangle.$ We then rewrite it as 
\begin{align}\label{cal}
	w_k \D( A \n w_k)&=\partial_k(w_k(H(w)+Vw))-w_{kk}(H(w)+Vw) - w_k \partial_t w_k  -\D(w_k(\partial_k A)\n w)\\ &+\langle \n w_k,(\partial_k A)\n w \rangle
	 -w_k V_kw - Vw_k^2\notag.
\end{align}
We now use \eqref{cal} in \eqref{ca} to get
\begin{align*}
	\partial_t w_k^2 + \D(A \n (w_k^2)) &=2w_k \partial_t w_k + 2\partial_k(w_k( H(w)+Vw))-2w_{kk}(H(w)+Vw) -2 w_k \partial_t w_k  -2\D(w_k(\partial_k A)\n w)\\
	& +2\langle \n w_k,(\partial_k A)\n w \rangle + 2 \langle A \n w_k, \n w_k\rangle-2w_k V_kw -2V w_k^2\\
	&=2\partial_k(w_k (H(w)+Vw))-2w_{kk}(H(w)+Vw)  -2\D(w_k(\partial_k A)\n w) +2\langle \n w_k,(\partial_k A)\n w \rangle \\
	&\;\;\;+ 2 \langle A \n w_k, \n w_k\rangle-2w_k V_kw -2V w_k^2.
\end{align*}
We can now rewrite the above identity as 
\begin{align*}
	2 \langle A \n w_k, \n w_k\rangle &= H(w_k^2)- 2\partial_k(w_k (H(w)+Vw))+2w_{kk}(H(w)+Vw)  +2\D(w_k(\partial_k A)\n w)\\
	 &-2\langle \n w_k,(\partial_k A)\n w \rangle +2w_kV_k w  + 2V w_k^2.
\end{align*}
We now multiply by $\s^{2-2\A}G$ and use integration by parts formula to obtain 
\begin{align}\label{a11}
&2 \sum_k	\int \s^{2-2\A}\langle A \n w_k, \n w_k\rangle G  = \int \s^{2-2\A} |\n w|^2 H^*(G) +(2\A-2)\int \s^{1-2\A}\s'|\n w|^2 G- \int_{\{t=a\}} \s^{2-2\A}|\n w|^2 G \\
&+2 \int \s^{2-2\A} \langle \n w, \n G\rangle (H(w)+Vw)+2\sum_k\int \s^{2-2\A}  w_{kk}(H(w)+Vw)G -2\sum_k\int \s^{2-2\A} w_k\langle (\partial_k A)\n w, \n G \rangle\notag\\
 &-2\sum_k \int \s^{2-2\A} \langle \n w_k,(\partial_k A)\n w \rangle G +2\int \s^{2-2\A}\langle \n w, \n V\rangle w G +2\int \s^{2-2\A}  V |\n w|^2G,\notag
\end{align}
where 	$H^*G=\D(A(x,0) \n G)-G_t.$
 On using the Cauchy-Schwarz inequality and $ |\n G| \le \frac{|x|}{4t}G  \leq  \frac{4}{t}G$ on the  right-hand side of \eqref{a11}, we find
\begin{align}\label{a12}
	&2 \sum_k	\int \s^{2-2\A}\langle A \n w_k, \n w_k\rangle G  \le \int \s^{2-2\A} |\n w|^2 H^*(G) +(2\A-2)\int \s^{1-2\A}\s'|\n w|^2 G- \int_{\{t=a\}} \s^{2-2\A}|\n w|^2 G \\
	&+ \int \s^{2-2\A} |\n w|^2 \frac{|x|^2}{16t^2}G+\int \s^{2-2\A} (H(w)+Vw)^2G
	+\frac{1}{4N}\sum_k\int \s^{2-2\A}  w_{kk}^2G +4nN\int \s^{2-2\A} (H(w)+Vw)^2G\notag\\
	&+N\int \s^{1-2\A} |\n w|^2G+\frac{1}{4N}\sum_k \int \s^{2-2\A}  |\n w_k|^2G+4N^2\int \s^{2-2\A} |\n w|^2G\notag\\
	&+\int \s^{2-2\A} |\n w|^2 G +\int \s^{2-2\A} |\n V|^2 w^2G+2\int \s^{2-2\A}  |V||\n w|^2G.\notag
\end{align}
We now use  $$\sum_k \langle A \n w_k, \n w_k\rangle \ge \frac{1}{N} |D^2w|^2,$$ $\s<1$, $\s'= O(1)$ and combine the like terms to get
\begin{align}\label{a121}
	&\frac{6}{4N}	\int \s^{2-2\A}|D^2 w|^2 G  \le \int \s^{2-2\A} |\n w|^2 H^*(G) +(2\A-1+N+4N^2)\int \s^{1-2\A}|\n w|^2 G- \int_{\{t=a\}} \s^{2-2\A}|\n w|^2 G \\
	&+ \int \s^{2-2\A} |\n w|^2 \frac{|x|^2}{16t^2}G+(1+4nN)\int \s^{2-2\A} (H(w)+Vw)^2G+\int \s^{2-2\A} |\n V|^2 w^2G+\int \s^{2-2\A}  |V||\n w|^2G.\notag
\end{align}
We take $\A$ large enough such that $\A>4N^2+N.$ Since we have $t<1/\lambda,$ $\s <t$ and $\lambda=\A/\delta^2$, therefore we get $\s \le \delta^2/\A.$ Consequently \eqref{a121} becomes 
\begin{align}\label{a22}
	&\frac{6}{4N}	\int \s^{2-2\A}|D^2 w|^2 G  \le \int \s^{2-2\A} |\n w|^2 H^*(G) +3\A\int \s^{1-2\A}|\n w|^2 G- \int_{\{t=a\}} \s^{2-2\A}|\n w|^2 G + \int \s^{2-2\A} |\n w|^2 \frac{|x|^2}{16t^2}G\\
	&+\frac{(1+4nN)\delta^2}{\A}\int \s^{1-2\A} (H(w)+Vw)^2G
	+\frac{\delta^4}{\A^2}\int \s^{-2\A}|\n V|^2w^2G+\frac{\delta^2}{\A}\int \s^{1-2\A}| V||\n w|^2G.\notag
\end{align}
We will now  take $\A$ larger if necessary  such that $\A \ge N ||V(\cdot, 0)||_1^{1/2}.$ Consequently $ ||\n V||_{\infty}^2 \le \A^4/N^4$ and $ ||V||_{\infty} \le \A^2/N^2.$ Hence \eqref{a22} becomes
\begin{align}\label{a23}
		&\frac{6}{4N}	\int \s^{2-2\A}|D^2 w|^2 G  \le \int \s^{2-2\A} |\n w|^2 H^*(G) +3\A\int \s^{1-2\A}|\n w|^2 G- \int_{\{t=a\}} \s^{2-2\A}|\n w|^2 G \\
	& +\int \s^{2-2\A} |\n w|^2 \frac{|x|^2}{t^2}G+\frac{(1+4nN)\delta^2}{\A}\int \s^{1-2\A} (H(w)+Vw)^2G
	 +\frac{\A^2\delta^4}{N^4}\int \s^{-2\A}w^2G +\frac{\A \delta^2}{N^2}\int \s^{1-2\A}|\n w|^2G .\notag
\end{align}
Next we note that \begin{align*}
	H^*G=\D(A \n G)-G_t&=\sum a_{ij}G_{ij}+ \sum \partial_j a_{ij} G_i -G_t\\
	&=\sum (a_{ij}- \delta_{ij}) G_{ij}+ \sum \partial_j a_{ij} G_i +\sum \delta_{ij}G_{ij} -G_t\\
	&=\sum (a_{ij}- \delta_{ij}) G_{ij}+ \sum \partial_j (a_{ij}) G_i,
\end{align*}
where we have used $G$ is a fundamental solution of heat eqaution. We now use $A$ is Lipschitz to get \begin{align*}
|	H^*G| &\le \sum |(a_{ij}- \delta_{ij})| |G_{ij}|+ \sum |\partial_j (a_{ij})| |G_i|\\
&\le N\left( \frac{|x|}{t}+ \frac{|x|^3}{t^2}\right).
\end{align*}
Hence using $\s\le t$ and $x \in B_4$ we have 
\begin{align}\label{a24}
	\int \s^{2-2\A} |\n w|^2 H^*(G)+ \int \s^{2-2\A} |\n w|^2 \frac{|x|^2}{t^2}G &\le N \int t \s^{1-2\A} |\n w|^2\left( \frac{|x|}{t}+ \frac{|x|^3}{t^2}+\frac{|x|^2}{t^2}\right)G\\
	&\le N \int  \s^{1-2\A} |\n w|^2\left( 1+ \frac{|x|^2}{t}\right)G\notag\\
	&\le \delta NN^{2\A}\lambda^{2\A+N}\int t|\n w|^2+ N \delta^2 \int \s^{1-2\A}\frac{\theta(\lambda t)}{t}|\n w|^2 G\notag\\
	& +N \int \s^{1-2\A}|\n w|^2 G,\notag
\end{align}
 where the last inequality follows from Lemma \ref{logi}. We  now use \eqref{a24} in \eqref{a23} to find for all large $\alpha$ 
 \begin{align}\label{a25}
 	&\frac{6}{4N}	\int \s^{2-2\A}|D^2 w|^2 G  \le \delta NN^{2\A}\lambda^{2\A+N}\int t|\n w|^2+ N \delta^2 \int \s^{1-2\A}\frac{\theta(\lambda t)}{t}|\n w|^2 G +4\A\int \s^{1-2\A}|\n w|^2 G\\
 	&- \int_{\{t=a\}} \s^{2-2\A}|\n w|^2 G +\frac{(1+4nN)\delta^2}{\A}\int \s^{1-2\A} (H(w)+Vw)^2G
 	+\frac{\A^2\delta^4}{N^4}\int \s^{-2\A}w^2G +\frac{\A \delta^2}{N^2}\int \s^{1-2\A}|\n w|^2G .\notag
 \end{align} Now for large enough $\alpha$, it follows from \eqref{a25}  by using  $\A=\lambda \delta^2$ and $\lambda \le \frac{\theta(\lambda t)}{t}$ that the following holds
 \begin{align}\label{a25}
 	&\frac{6}{4N}	\int \s^{2-2\A}|D^2 w|^2 G  \le \delta NN^{2\A}\lambda^{2\A+N}\int t|\n w|^2+ (N+4+\delta^2/N^2)\delta^2 \int \s^{1-2\A}\frac{\theta(\lambda t)}{t}|\n w|^2 G \\
 	&- \int_{\{t=a\}} \s^{2-2\A}|\n w|^2 G +\frac{(1+4nN)\delta^2}{\A}\int \s^{1-2\A} (H(w)+Vw)^2G
 	+\frac{\A^2\delta^4}{N^4}\int \s^{-2\A}w^2G.\notag
 \end{align}
After taking $N$ large enough we get 
\begin{align}\label{a25}
	&	\int \s^{2-2\A}|D^2 w|^2 G  \le \delta NN^{2\A}\lambda^{2\A+N}\int t|\n w|^2+ N\delta^2 \int \s^{1-2\A}\frac{\theta(\lambda t)}{t}|\n w|^2 G \\
	&- \int_{\{t=a\}} \s^{2-2\A}|\n w|^2 G +N \delta^2\int \s^{1-2\A} (H(w)+Vw)^2G
	+\A^2\delta^4\int \s^{-2\A}w^2G.\notag
\end{align}
Now by replacing $w(\cdot, t)$  by   $w(\cdot,t-a)$   in \eqref{a25} and then by letting $t-a$ as our new $t$ we find
\begin{align}\label{d2f}
	\int \s_a^{2-2\A}|D^2w|^2 G_a &\le   N \delta^2 \int \s_a^{1-2\A}(\D(A(x,0) \n w)+w_t+Vw)^2 G_a +\delta N N^{2\A}\lambda^{2\A+N}\underset{t \ge 0}{\operatorname{sup}}\int t|\n w|^2 dx\\
	&+ N\delta^2 \int \s_a^{1-2\A}\frac{\theta(\lambda t)}{t}|\n w|^2 G_a
	+\A^2\delta^4\int \s_a^{-2\A}w^2G_a.\notag
\end{align}
Now we have the estimate for $	\int \s_a^{2-2\A}|D^2w|^2 G_a.$
Therefore from \eqref{dis} and \eqref{d2f} it follows
\begin{align}\label{br1}
		&N \int \s_a^{1-2\A}(\D((A(x,0)-A(x,t)) \n w)^2 G_a\\
		&\le 	NM^2 \int \s_a^{1-2\A} |\n w|^2 G_a
		 +N^3M^2 \delta^2 \int \s_a^{1-2\A}(\D(A(x,0) \n w)+w_t+Vw)^2 G_a\notag\\
		 & +\delta N^3M^2 N^{2\A}\lambda^{2\A+N}\underset{t \ge 0}{\operatorname{sup}}\int t|\n w|^2 dx+ N^3M^2\delta^2 \int \s_a^{1-2\A}\frac{\theta(\lambda t)}{t}|\n w|^2 G_a
		 +\A^2\delta^4N^2M^2\int \s_a^{-2\A}w^2G_a.\notag
\end{align}
We now plug the estimate \eqref{br1} in \eqref{br} to obtain
\begin{align}
	&	N \int \s_a^{1-2\A}(\D(A(x,0) \n w)+w_t+Vw)^2G_a\\
		 &\le 	2NM^2 \int \s_a^{1-2\A} |\n w|^2 G_a
		+2N^3M^2 \delta^2 \int \s_a^{1-2\A}(\D(A(x,0) \n w)+w_t+Vw)^2 G_a\notag\\
		& +2\delta N^3M^2 N^{2\A}\lambda^{2\A+N}\underset{t \ge 0}{\operatorname{sup}}\int t|\n w|^2 dx+ 2N^3M^2\delta^2 \int \s_a^{1-2\A}\frac{\theta(\lambda t)}{t}|\n w|^2 G_a
		+2\A^2\delta^4N^2M^2\int \s_a^{-2\A}w^2G_a\notag\\
		&+2 	N \int \s_a^{1-2\A}(\D(A(x,t) \n w)+w_t+Vw)^2G_a.\notag
\end{align}
We now choose $\delta$ small enough such that $N>4N^3M^2\delta^2$ to deduce the following inequality from above
\begin{align}\label{brf}
	&	\frac{N}{2} \int \s_a^{1-2\A}(\D(A(x,0) \n w)+w_t+Vw)^2G_a\\
	&\le 	2NM^2 \int \s_a^{1-2\A} |\n w|^2 G_a
+2\delta N^3M^2 N^{2\A}\lambda^{2\A+N}\underset{t \ge 0}{\operatorname{sup}}\int t|\n w|^2 dx+ 2N^3M^2\delta^2 \int \s_a^{1-2\A}\frac{\theta(\lambda t)}{t}|\n w|^2 G_a\notag\\
	&+2\A^2\delta^4N^2M^2\int \s_a^{-2\A}w^2G_a
	+2 	N \int \s_a^{1-2\A}(\D(A(x,t) \n w)+w_t+Vw)^2G_a.\notag
\end{align}
Using \eqref{brf} in \eqref{step1} ( more precisely in a  a "shifted in time" version of the estimate \eqref{step1} as in \eqref{d2f}) we obtain
\begin{align}
		&\A^2\int \s_a^{-2\A}w^2G_a+ \int \s_a^{1-2\A}|\n w|^2\frac{\theta(\lambda t)}{t} G_a\\
	&\le  	4NM^2 \int \s_a^{1-2\A} |\n w|^2 G_a
	+4\delta N^3M^2 N^{2\A}\lambda^{2\A+N}\underset{t \ge 0}{\operatorname{sup}}\int t|\n w|^2 dx+ 4N^3M^2\delta^2 \int \s_a^{1-2\A}\frac{\theta(\lambda t)}{t}|\n w|^2 G_a\notag\\
	&+4\A^2\delta^4N^2M^2\int \s_a^{-2\A}w^2G_a+4N\int \s_a^{1-2\A}\left(\D(A(x,t) \n w)+w_t + V w\right)^2G_a  
	+NN^{2\A}\lambda^{2\A+N}\underset{t\ge 0}{\operatorname{sup}}\int w^2+t|\n w|^2\notag\\
	& -\frac{1}{N}\int t\s_a^{-2\A}|\n w(x,0)|^2 G_adx  +N\A\int \s_a^{-2\A}w^2(x,0) G_adx.\notag
\end{align}
After rearranging the terms in the above inequality we find 
\begin{align}
	&(\A^2-4\A^2\delta^4N^2M^2)\int \s_a^{-2\A}w^2G_a+ (1-4N^3M^3\delta^2)\int \s_a^{1-2\A}|\n w|^2\frac{\theta(\lambda t)}{t} G_a-	4NM^2 \int \s_a^{1-2\A} |\n w|^2 G_a\\
	&\le  4\delta N^3M^2 N^{2\A}\lambda^{2\A+N}\underset{t \ge 0}{\operatorname{sup}}\int t|\n w|^2 dx
	+4N\int \s_a^{1-2\A}\left(\D(A(x,t) \n w)+w_t + V w\right)^2G_a\notag\\  
	&+NN^{2\A}\lambda^{2\A+N}\underset{t\ge 0}{\operatorname{sup}}\int w^2+t|\n w|^2
	-\frac{1}{N}\int t\s_a^{-2\A}|\n w(x,0)|^2 G_adx  +N\A\int \s_a^{-2\A}w^2(x,0) G_adx.\notag
\end{align}
We now  choose $\delta$  smaller if necessary  such that $1-4\delta^4N^2M^2>1/2$ and $1-4N^3M^3\delta^2>1/2$ holds. Consequently, we get for a new constant $N$
\begin{align}\label{f11}
&	\A^2 \int \s_a^{-2\A}w^2G_a + \int \s_a^{1-2\A}|\n w|^2 \frac{\theta(\lambda t)}{t}G_a-N \int \s_a^{1-2\A}|\n w|^2 G_a\\
&\le N \int \s_a^{1-2\A}(\D(A(x,t) \n w)+w_t+V(x,0)w)^2 G_a +NN^{2\A}\lambda^{2\A+N}\underset{t \ge 0}{\operatorname{sup}}\int w^2 +|\n w|^2 dx\notag\\
	&+\s(a)^{-2\A}\left(-\frac{a}{N}\int |\n w (x,0)|^2G(x,a)dx + N \A \int w^2(x,0)G(x,a)dx\right).\notag
\end{align}
Since $\frac{\theta(\lambda t)}{t} \le \lambda$ therefore by choosing $\A$ large enough such that $\A \ge 2N$ we get 
\begin{align}\label{f12}
	&	\A^2 \int \s_a^{-2\A}w^2G_a + \frac{\A}{2}\int \s_a^{1-2\A}|\n w|^2 G_a\\
	&\le N \int \s_a^{1-2\A}(\D(A(x,t) \n w)+w_t+V(x,0)w)^2 G_a +NN^{2\A}\lambda^{2\A+N}\underset{t \ge 0}{\operatorname{sup}}\int w^2 +|\n w|^2 dx\notag\\
	&+\s(a)^{-2\A}\left(-\frac{a}{N}\int |\n w (x,0)|^2G(x,a)dx + N \A \int w^2(x,0)G(x,a)dx\right).\notag
\end{align}
\eqref{cari} follows from \eqref{f12} in a standard way. 

\textbf{Step 3}: (Conclusion)  We finally show that exists a universal constant $N$ and $\delta$ such that for all $\A \ge N(1+||V||_1^{1/2}+[V]_{1/2}^{1/2})$ the following inequality holds
\begin{align}\label{cari1}
	&	\A^2 \int \s_a^{-2\A}w^2G_a +\A \int \s_a^{1-2\A}|\n w|^2G_a\\
	& \le N \int \s_a^{1-2\A}(\D(A(x,t) \n w)+w_t+V(x,t)w)^2 G_a +N^{2\A}\A^{2\A}\underset{t \ge 0}{\operatorname{sup}}\int w^2
	+t|\n w|^2 dx\notag\\
	&+\s(a)^{-2\A}\left(-\frac{a}{N}\int |\n w (x,0)|^2G(x,a)dx + N \A \int w^2(x,0)G(x,a)dx\right).\notag
\end{align}
 for all $0<a\le \frac{1}{4\lambda}$ and $w \in C_{0}^{\infty}(B_4 \times [0,\frac{1}{2\lambda}))$, where $\lambda=\A/\delta^2.$
 
 \medskip
 
\textbf{Proof of Step 3}: We now  replace $V(x,0)$ by $V(x,t).$ Using AM-GM inequality  we have
\begin{align*}
	|(\D(A(x,t) \n w)+w_t+V(x,0)w)|&=|(\D(A(x,t) \n w)+w_t+V(x,t)w)+(V(x,0)-V(x,t))w|\\
	&\le 2(\D(A(x,t) \n w)+w_t+V(x,t)w)^2+2(V(x,0)-V(x,t))^2w^2.
\end{align*}
Since $V(x,t)$ is $1/2$-H\"{o}lder continuous, $(V(x,0)-V(x,t))^2 \le [V]_{1/2}^2t.$  Thus we get
\begin{align}\label{v1}
 N \int \s_a^{1-2\A}(\D(A(x,t) \n w)+w_t+V(x,0)w)^2 G_a &\le  2N \int \s_a^{1-2\A}(\D(A(x,t) \n w)+w_t+V(x,t)w)^2 G_a\\&+2 N[V]_{1/2}^2 \int \s_a^{1-2\A}tw^2 G_a.\notag
\end{align}
We would be done if  we can absorb the last term in the right-hand side of \eqref{v1}  into  the left-hand side of \eqref{cari}. Using $\lambda t \le 1/2$  and $\s_a(t) \le t+a \le 1/\lambda$, we get
\begin{align}\label{kv1}
	2 N[V]_{1/2}^2 \int \s_a^{1-2\A}tw^2 G_a \le 	\frac{2 N}{\lambda^2}[V]_{1/2}^2 \int \s_a^{-2\A}w^2 G_a
\end{align}
Using \eqref{v1} and \eqref{kv1} in \eqref{cari} we obtain
\begin{align}\label{f2}
	&	\A^2 \int \s_a^{-2\A}w^2G_a +\A \int \s_a^{1-2\A}|\n w|^2G_a\\
	& \le N \int \s_a^{1-2\A}(\D(A(x,t) \n w)+w_t+V(x,t)w)^2 G_a\notag\\
	&+\frac{8 N}{\A^2}[V]_{1/2}^2 \int \s_a^{-2\A}w^2 G_a +N^{\A}\A^{2\A}\underset{t \ge 0}{\operatorname{sup}}\int w^2
	 +|\n w|^2 dx\notag\\
	&+\s(a)^{-2\A}\left(-\frac{a}{N}\int |\n w (x,0)|^2G(x,a)dx + N \A \int w^2(x,0)G(x,a)dx\right).\notag
\end{align}
Now note that the second term in the right-hand side of \eqref{f2} can be absorbed in the first term of left-hand side of \eqref{f2} provided $\frac{\A^2}{2} \ge \frac{8 N}{\alpha^2}[V]_{1/2}^2$ i.e. when $\A \ge 4N [V]_{1/2}^{1/2}.$ Hence if we choose $\A \ge N(||V||_1^{1/2}+ [V]_{1/2}^{1/2}+1)$ for $N$ large, the conclusion follows.

	\end{proof}
	In order to establish a quantitative space-like doubling property that implies \eqref{df}, we also need the following monotonicity in time result. The following lemma below is analogous to Lemma 5 in \cite{EFV}. 
\begin{lemma}\label{mon}
	Let $u$ be a solution to \eqref{meq} in $Q_4$.  Then there exists a universal constant $N$ such that the following inequality holds
	\begin{align}\label{moni}
	N e^{ ||V||_{\infty}^{1/2}}\int_{B_{2\rho}} u^2(x,t)dx \ge \int_{B_{\rho}} u^2(x,0)dx
	\end{align}
 for $t\le \frac{\rho^2}{2N\operatorname{log}(2N(1+||V||_{\infty}){\Theta_{\rho}})+ 5 N^2 (||V||_{\infty}^{1/2}+1)}$ and $0<\rho\le1$, where $$\Theta_{\rho}=\frac{\int_{Q_{4}} u^2( x, t)}{\rho^2 \int_{B_{\rho}}u^2( x,0)}.$$ 
\end{lemma}
\begin{proof}
The proof is although similar  to that in \cite{EFV}, we nevertheless provide the  details  in order to highlight the delicate dependence  on $||V||_{\infty}$  in \eqref{moni} above. 
	We consider $f(x,t)= u(x,t) \phi(x),$ where $\phi=1$ in $B_{3/2}$ and $\phi =0$ outside $B_2.$  We then define $$H(t)= \int f^2(x,t) \mathcal G(x,t;y,0) dx,$$ where $\mathcal G(x,t;y,s)$ is the fundamental solution in $\R^{n+1}$ of the operator $\partial_t - \operatorname{div}(A(x,t) \nabla)$, i.e. $\mathcal G_t- \D(A \n \mathcal G)=\delta_{(y,s)}.$  Over here, we would like to mention that without loss of generality, we may assume that $A$ is defined on whole of $\R^n \times \R$ and satisfies  the bounds in \eqref{assump} and moreover the associated fundamental solution $\mathcal G$ satisfies
	\begin{equation}
	\begin{cases}
	\int_{\R^n} \mathcal G(x,t;y, 0) dy=1\\
	\int_{\R^n} \mathcal G(x,t;y,0) \phi(y) dy \to \phi(x)\ \text{as $t \to 0$.}\end{cases}
	\end{equation}
	See  for instance \cite[Chapter 1]{Fr}.
	
	Then for $t>0$ we have
	\begin{align*}
		H'(t) &= \int 2ff_t \mathcal G +\int f^2 \mathcal G_t \\
		&=\int 2ff_t \mathcal G + \int f^2 \D(A \n \mathcal G)\\
		&=\int 2ff_t \mathcal G - \int\langle \n f^2,A \n \mathcal G \rangle \\
		&=\int 2ff_t \mathcal G + \int \D(A \n f^2)\mathcal G\\
		&=2\int f(f_t+ \D(A \n f))\mathcal G + \int \langle A \n f, \n f \rangle \mathcal G.
	\end{align*}
 From straightforward calculations we have $f_t=u_t \phi $ and 
$$ \D(A \n f) = \phi \D (A \n u) + 2 \langle A \n u, \n \phi \rangle + u \D(A \n \phi ).$$
Since $u$ is a solution of \eqref{meq}, we get $f_t + \D(A \n f)=- Vf +   u\D(A \n \phi ) +  2 \langle A \n u, \n \phi \rangle.$ 	
Hence we get \begin{align*}
	H'(t) \ge   - N||V||_{\infty} H(t) -  2N\int |f|(|u|+|\n u|)\chi_{B_2\setminus B_{3/2}}\mathcal G.
\end{align*}

	  Arguing as in \cite{EFV} we get
	 \begin{align*}
		H'(t) \ge - N||V||_{\infty} H(t) - \frac{N}{t^{n/2}}e^{-\frac{1}{4Nt}}\int (|u|^2+|\n u|^2).
	\end{align*}
Then using  the estimate in \eqref{eniq}, we  obtain from above
\begin{align*}
	H'(t) \ge - N||V||_{\infty} H(t)  -(1+||V||_{\infty})\frac{N}{t^{n/2}}e^{-\frac{1}{4Nt}}\int_{Q_4} |u|^2.
\end{align*}
Now for sufficiently small $t$, one has  $t^{n/2}e^{-\frac{1}{4Nt}} \le  e^{-\frac{1}{8Nt}}$ and hence we deduce the following inequality for all large enough $N$
\begin{align*}
	H'(t) \ge - N||V||_{\infty} H(t)  -(1+||V||_{\infty})Ne^{-\frac{1}{Nt}}\int_{Q_4} u^2.
\end{align*}
Integrating this inequality in $(0,t)$ and using $\underset{t \rightarrow 0}{\operatorname{lim}} H(t)=u^2(y,0)$ for $y \in B_1$ we get
\begin{align}\label{intp}
	e^{N||V||_{\infty} t} \int f^2(x,t) \mathcal G(x,t;y,0) \ge u^2(y,0) - e^{N||V||_{\infty} t} (1+||V||_{\infty})Nte^{-\frac{1}{Nt}}\int_{Q_4} u^2.
\end{align}
Now by integrating the above  inequality in \eqref{intp} in  the $y$-variable in $B_1$,  changing  the order of integration and then by  using $\int \mathcal G(x,t;y,0)dy=1$ we obtain
\begin{align}\label{monp}
e^{N||V||_{\infty} t}\int_{B_2} u^2(x,t) \ge	e^{N||V||_{\infty} t} \int_{B_2} f^2(x,t) \ge \int_{B_1}u^2(x,0) -e^{N ||V||_{\infty} t} (1+||V||_{\infty})Nte^{-\frac{1}{Nt}}\int_{Q_4} u^2.
\end{align}
We now choose a universal $N$ (which does not depend on $||V||_{\infty}$) such that $2N/C_E >1$ and $ N \operatorname{log}(2N/C_E)>1,$ where $C_E$ is from Lemma \ref{enl}.
Note that from \eqref{eniq}  we have $$1 \le C_E(1+||V||_{\infty})\frac{\int_{Q_4} u^2dX}{\int_{B_1}u^2(x,0)dx}=C_E(1+||V||_{\infty})\Theta.$$  Hence we get $N\operatorname{log}(2N(1+||V||_{\infty})\Theta)= N\operatorname{log}(2N/C_E)+N\operatorname{log}(C_E(1+||V||_{\infty})\Theta)>1.$ If we now take $t \le \frac{1}{2N\operatorname{log}(2N(1+||V||_{\infty})\Theta) + 5 N^2 (||V||_{\infty}^{1/2}+1)},$ then 
\begin{align}\label{mon2}
	e^{N||V||_{\infty} t} (1+||V||_{\infty})Ne^{-\frac{1}{Nt}}\int_{Q_4} u^2 \le \frac{1}{2} \int_{B_1}u^2(x,0).
\end{align}
 Using \eqref{mon2} in \eqref{monp}, we obtain  for all  $t \le \frac{1}{2N\operatorname{log}(2N(1+||V||_{\infty})\Theta) + 5 N^2 (||V||_{\infty}^{1/2}+1)},$ that the following inequality holds
 \begin{align*}
	N e^{ ||V||_{\infty}^{1/2}}\int_{B_2} u^2(x,t) \ge \int_{B_1}u^2(x,0).
\end{align*}
Thus, for all $t \le\frac{1}{2N\operatorname{log}(2N(1+||V||_{\infty})\Theta) + 5 N^2 (||V||_{\infty}^{1/2}+1)}$ we have 
\begin{align*}
N e^{ ||V||_{\infty}^{1/2}}\int_{B_2} u^2(x,t) \ge \int_{B_1}u^2(x,0),
\end{align*}
where $\Theta=\frac{\int_{Q_4} u^2dX}{\int_{B_1}u^2(x,0)dx}.$
Now for the rescaled solution $\tilde{u}(x,t):=u(\rho x, \rho^2 t),$ $\rho<1$, we get
\begin{align*}
N e^{ ||V||_{\infty}^{1/2}}\int_{B_{2\rho}} u^2(x,\rho^2 t) \ge \int_{B_{\rho}}u^2(x,0)
\end{align*}
where $t \le \frac{1}{2N\operatorname{log}(2N(1+||V||_{\infty})\tilde{\Theta_{\rho}})+ 5 N^2 (||V||_{\infty}^{1/2}+1)} $ and $$\tilde{\Theta_{\rho}}=\frac{\int_{Q_4} u^2(\rho x, \rho^2 t)dX}{\int_{B_1}u^2(\rho x,0)dx}.$$
After change of variables and by using  $\int_{Q_{4\rho}} u^2( x, t) \le \int_{Q_4} u^2(x, t)$ we get 
$$\tilde{\Theta_{\rho}}\le \frac{\int_{Q_{4}} u^2( x, t)}{\rho^2 \int_{B_{\rho}}u^2( x,0)}:=\Theta_{\rho}.$$ 
Thus, for all $t \le \frac{1}{2N\operatorname{log}(2N(1+||V||_{\infty}){\Theta_{\rho}})+ 5 N^2 (||V||_{\infty}^{1/2}+1)} $ we have
\begin{align}\label{inter1}
N e^{ ||V||_{\infty}^{1/2}} \int_{B_{2\rho}} u^2(x,\rho^2 t) \ge \int_{B_{\rho}}u^2(x,0).
\end{align}
From \eqref{inter1} it follows that 
 for all $t \le\frac{\rho^2}{2N\operatorname{log}(2N(1+||V||_{\infty}){\Theta_{\rho}})+ 5 N^2 (||V||_{\infty}^{1/2}+1)}$ we have
\begin{align*}
N e^{||V||_{\infty}^{1/2}}\int_{B_{2\rho}} u^2(x, t) \ge \int_{B_{\rho}}u^2(x,0),
\end{align*}
which completes the proof of the lemma.
	\end{proof}
	
	With the Carleman estimate in Lemma \ref{carleman}   and the monotonicity result  Lemma \ref{mon} in hand, we now proceed with the proof of Theorem \ref{main} where we adapt some ideas from \cite{EFV} in order to obtain our desired quantitative uniqueness estimate \eqref{df}.
\begin{proof}[Proof of Theorem \ref{main}]  Without loss of generality, we assume that $A(0,0)= \mathbb I_n$. The proof is divided into three steps.   \\
	 \textbf{Step 1 }: We first show that there exists a universal constant $N$ such that for all $r<1/2$
	 \begin{align}
	 	\int_{B_{2r}} u^2(x,0)dx \le M \int_{B_r} u^2(x,0)dx,
	 \end{align}
  where $M=\exp(N^4 \operatorname{log}(2N(1+||V||_{\infty})\Theta_{\rho})+N^4(||V||_1^{1/2}+[V]_{1/2}^{1/2}+1))$ and where $\Theta_{\rho}$ is as in Lemma \ref{mon}.\\
  \textit{Proof of Step 1}:
	For a fixed $\alpha$ large, with $\lambda=\frac{\alpha}{\delta^2}$, consider $w(x,t)=u(x,t)\psi(x) \phi(t),$ where $\psi \in C^{\infty}_{0}(B_4)$ such that $\psi \equiv 1$ in $B_3$,  $\psi =0$ outside $B_{7/2}$ and   where $\phi \equiv 1$ in $0\le t \le 1/4\lambda$ and $\phi =0$ for $t \ge 1/2\lambda$. Then using \eqref{meq} we have
	$$ \D(A\n w)+ w_t+Vw = \D(A \n \psi)u \phi + 2 \langle A \n \psi, \n u \rangle \phi  + u \psi \phi_t.$$
	It thus follows	\begin{align}\label{1st}
		(\D(A\n w)+ w_t+Vw)^2 \le C (u^2 +|\n u|^2)\chi_{B_3\setminus B_{2}}(x) +C\lambda^2u^2 \chi_{(0,1/2\lambda)\setminus (0,1/4\lambda)}(t). 
	\end{align}
	Using the Carleman inequality \eqref{carl} we get
	\begin{align}
		\A^2 \int \s_a^{-2\A}w^2G_a +\A \int \s_a^{1-2\A}|\n w|^2G_a \le N \int \s_a^{1-2\A}(\D(A \n w)+w_t+Vw)^2 G_a +N^{2\A}\A^{2\A}\underset{t \ge 0}{\operatorname{sup}}\int w^2 +t|\n w|^2 dx\\
		+\s(a)^{-2\A}\left(-\frac{a}{N}\int |\n w (x,0)|^2G(x,a)dx + N \A \int w^2(x,0)G(x,a)dx\right),\notag
	\end{align}
which  in view of \eqref{1st} implies the following estimate
\begin{align}\label{est2}
	&\A^2 \int \s_a^{-2\A}w^2G_a +\A \int \s_a^{1-2\A}|\n w|^2G_a\\
	 &\le N \int \s_a^{1-2\A}(u^2 +|\n u|^2)\chi_{B_3\setminus B_{2}}(x) G_a +N \lambda^2\int \s_a^{1-2\A}u^2 \chi_{(0,1/2\lambda)\setminus (0,1/4\lambda)}(t)G_a\notag\\ 
	 &+N^{\A}\A^{2\A}\underset{t \ge 0}{\operatorname{sup}}\int w^2 +t|\n w|^2 dx
	+\s(a)^{-2\A}\left(-\frac{a}{N}\int |\n w (x,0)|^2G(x,a)dx + N \A \int w^2(x,0)G(x,a)dx\right).\notag
\end{align}
We now make the following claim.

\medskip

	\emph{Claim}:  $ \s_a^{1-2\A}G_a \le N^{2\A+\frac{n}{2}} \lambda^{2\A+\frac{n}{2}}$ in the region $B_3 \times [0,1/2\lambda] \setminus B_2 \times [0,1/4\lambda].$\\

	Proof of the claim: We will divide the proof in two cases.\\
	Case(i): When $t> 1/4\lambda.$\\
	 From Lemma \ref{sig}, $\s_a \ge \frac{t+a}{N},$ which gives $\s_a^{1-2\A} \le N^{2\A-1}(t+a)^{1-2\A}.$ Thus we get
	  \begin{align}\label{c1}
	 	\s_a^{1-2\A}G_a \le  N^{2\A-1}(t+a)^{1-2\A} (t+a)^{-n/2}e^{-|x|^2/4(t+a)}.
 	\end{align}
	Since we have $t+a> 1/\lambda,$ which implies $(t+a)^{1-2\A-n/2} \le \lambda^{-1+2\A+n/2}.$ Also  since $(t+a)>0$  and thus  $e^{-|x|^2/4(t+a)} \le 1.$ Hence we obtain $$\s_a^{1-2\A}G_a \le  N^{2\A-1}\A^{-1+2\A+n/2},$$ which completes the proof of Case(i).\\
	Case(ii): When  $t \le  1/4\lambda.$\\ 
	Since $(x,t) \in B_3 \times [0,1/2\lambda] \setminus B_2 \times [0,1/4\lambda],$ therefore in this case we  must have $|x|\ge 2.$ Now we use \eqref{c1} to obtain
	\begin{align*}
		\s_a^{1-2\A}G_a \le  N^{2\A-1}(t+a)^{1-2\A} (t+a)^{-n/2}e^{-9/4(t+a)}
	\end{align*}
From the properties of exponential function, for any integer $k$ we have $e^{-x} \le \frac{k!}{x^k} \le  \frac{k^k}{x^k}.$ Thus for real number $2\A+n/2$ we get $e^{-1/(t+a)} \le (2\A+n/2)^{(2\A+n/2)} (t+a)^{2\A+n/2-1}. $ Hence we get $\s_a^{1-2\A}G_a \le N^{2\A-1}(2\A+n/2)^{(2\A+n/2)}.$ This proves the claim.\\
Using the claim proved above, we find that  \eqref{est2} becomes 
\begin{align}\label{est3}
	&\A^2 \int \s_a^{-2\A}w^2G_a +\A \int \s_a^{1-2\A}|\n w|^2G_a
		\le N^{2\A+n/2}\A^{2\A+n/2} \int_{0}^{2/\A}\int_{B_3} (u^2 +|\n u|^2) \\
		&+N^{2\A}\A^{2\A}\underset{t \ge 0}{\operatorname{sup}}\int w^2 +t|\n w|^2 dx
		+\s(a)^{-2\A}\left(-\frac{a}{N}\int |\n w (x,0)|^2G(x,a)dx + N \A \int w^2(x,0)G(x,a)dx\right).\notag
\end{align}
Now using the  estimate  in  \eqref{eniq} we then obtain
\begin{align}\label{est4}
	&\A^2 \int \s_a^{-2\A}w^2G_a +\A \int \s_a^{1-2\A}|\n w|^2G_a
	\le 2C_E N^{2\A+n/2}\lambda^{2\A+n/2}(1+||V||_{\infty})\int_{Q_4}u^2(x,t)dX\\
	&+\s(a)^{-2\A}\left(-\frac{a}{N}\int |\n w (x,0)|^2G(x,a)dx + N \A \int w^2(x,0)G(x,a)dx\right).\notag
\end{align}
Let $\rho \in (0,1)$ to be fixed later.  Now since $\phi(t)=1$ for $t \le 1/4\lambda$ and $\psi=1$ in $B_{2},$  we observe that the first term on the left hand side of \eqref{est4}  can be minorized as follows
\begin{align*}
	\A^2 \int \s_a^{-2\A}w^2G_a &\ge \A^2 \int_0^{1/4\lambda}\int_{B_3} \s_a^{-2\A}w^2G_a\\
	& \ge \A^2 \int_0^{\rho^2/4\lambda}\int_{B_{2\rho}} \s_a^{-2\A}u^2(t+a)^{-n/2}e^{-|x|^2/4(t+a)}\\
	&\ge \A^2 \int_0^{\rho^2/4\lambda}\int_{B_{2\rho}} (t+a)^{-2\A-n/2}u^2e^{-\rho^2/(t+a)},
\end{align*}
where the second inequality is a consequence of the fact that  $0<\rho<1$ and in the last inequality, we have used $\s_a(t) \le t+a.$ We now assume  that \begin{equation}\label{ch1} \A > N( \operatorname{log}(2N(1+||V||_{\infty})\Theta_{\rho})+ ||V||_{1}^{1/2}+[V]_{1/2}^{1/2}+1) \end{equation} and $0<t \le \frac{\rho^2}{4\lambda}\le \frac{\rho^2}{2\A}$ in order to apply Lemma \ref{mon} and consequently we obtain
\begin{align*}
		\A^2 \int \s_a^{-2\A}w^2G_a &\ge   \frac{\A^2}{N} \int_0^{\rho^2/4\lambda}(t+a)^{-2\A-n/2}e^{-\rho^2/(t+a)}dt\  e^{- ||V||_{\infty}^{1/2}} \int_{B_{ \rho}} u^2(x,0)dx\notag\\
		& \geq  e^{-\alpha} \frac{\A^2}{N} \int_a^{a+\rho^2/4\lambda}t^{-2\A-n/2}e^{-\rho^2/t}dt\int_{B_{ \rho}} u^2(x,0)dx\notag\\
		&\ge e^{-\alpha}  \frac{\A^2}{N} \int_{\rho^2/8\lambda}^{\rho^2/4\lambda}t^{-2\A-n/2}e^{-\rho^2/t}dt\int_{B_{ \rho}} u^2(x,0)dx,
\end{align*}
where in the second inequality above, we used that $e^{- ||V||_{\infty}^{1/2}} \geq e^{-\alpha}$ which can be ensured by taking $N>1$  in \eqref{ch1} above  and  the last inequality uses $0<a \le \frac{\rho^2}{8\lambda}$  which in turn implies $(\rho^2/8\lambda,\rho^2/4\lambda) \subset (a,a+\rho^2/4\lambda).$ Now we use the fact that $t^{-2\A-n/2}$ is a decreasing function of $t$ and $e^{-\rho^2/t}$ is an increasing function of $t$ to deduce the following inequality from the above estimate 
\begin{align}
		\A^2 \int \s_a^{-2\A}w^2G_a \ge \frac{\A^2\rho^2}{8\lambda N} \left(\frac{\rho^2}{4\lambda}\right)^{-2\A-\frac{n}{2}}e^{-10\lambda}\int_{B_{ \rho}} u^2(x,0)dx\ge \frac{\delta^2 4^{2\A+\frac{n}{2}}\lambda^{2\A+\frac{n}{2}+1}}{8 N}(e^{5/\delta^2}\rho^2)^{-2\A}\rho^{2-n}\int_{B_{ \rho}} u^2(x,0)dx,
\end{align}
where we have used  that $\lambda =\A/\delta^2$ and $e^{-\alpha} \geq e^{-2 \lambda}$.
We now recall the definition of $\Theta_{\rho}$ to find
\begin{align}\label{les}
	\A^2 \int \s_a^{-2\A}w^2G_a \ge \frac{\delta^2 4^{2\A+\frac{n}{2}}\lambda^{2\A+\frac{n}{2}+1}}{8 N\Theta_{\rho}}(e^{5/\delta^2}\rho^2)^{-2\A}\rho^{-n} \int_{Q_4} u^2(x,t)dX.
\end{align}
Now using \eqref{les} in \eqref{est4} we get
\begin{align}\label{g1}
 \frac{\delta^2 4^{2\A+\frac{n}{2}}\lambda^{2\A+\frac{n}{2}+1}}{8 N\Theta_{\rho}}(e^{5/\delta^2}\rho^2)^{-2\A}\rho^{-n} \int_{Q_4} u^2(x,t)dX  \le  2C_E N^{2\A+n/2}\lambda^{2\A+n/2}(1+||V||_{\infty})\int_{Q_4}u^2(x,t)dX\\
+\s(a)^{-2\A}\left(-\frac{a}{N}\int |\n w (x,0)|^2G(x,a)dx + N \A \int w^2(x,0)G(x,a)dx\right)\notag.
\end{align}
Now in order  to absorb the first term on  the right hand side of \eqref{g1} onto the  left hand side of \eqref{g1}, we choose $\rho$ such that
\begin{align}\label{ro1}
	\frac{1}{2}\frac{\delta^2 4^{2\A+\frac{n}{2}}\lambda^{2\A+\frac{n}{2}+1}}{8 N\Theta_{\rho}}(e^{5/\delta^2}\rho^2)^{-2\A}\rho^{-n}\ge 2C_E N^{2\A+n/2}\lambda^{2\A+n/2}(1+||V||_{\infty})
\end{align}
Since  $1+||V||_{\infty} \le \A^2$,  thus \eqref{ro1}  is ensured provided  for a possibly larger $N$ the following holds
\begin{align}\label{ro2}
		\frac{1}{2}\frac{\delta^2 4^{2\A+\frac{n}{2}}\lambda^{2\A+\frac{n}{2}+1}}{8 N\Theta_{\rho}}(e^{5/\delta^2}\rho^2)^{-2\A}\rho^{-n}\ge 2C_E N^{2\A+n/2}\lambda^{2\A+n/2+2}.\end{align}

\eqref{ro2} in particular will follow in case the following inequality holds for a larger $N$ \begin{align}\label{ro3}
N^{-2\A}(e^{5/\delta^2}\rho^2)^{-2\A} \ge 32C_EN^{n/2+1} \Theta_{\rho}.
\end{align} 
Since $\A \ge N \operatorname{log}(2N(1+||V||_{\infty})\Theta_{\rho})+N(||V||_1^{1/2}+[V]_{1/2}^{1/2}+1),$ we have $e^{\A /N} \ge 2N \Theta_{\rho}.$ Therefore \eqref{ro3} will hold if $$(Ne^{5/\delta^2}\rho^2)^{-2\A} \ge  16C_EN^{n/2}e^{\A/N}.$$ The above inequality can be rewritten as $$(e^{1/2N}Ne^{5/\delta^2}\rho^2)^{-2\A} \ge  16C_EN^{n/2}.$$ We now choose  $\rho$ such that \begin{equation}\label{cho}e^{1/2N}Ne^{5/\delta^2}\rho^2 \leq \frac{1}{8}\end{equation} and then we find that \eqref{ro1} follows. Using \eqref{ro1} in \eqref{g1} we obtain
\begin{align}
	\frac{a}{N}\int |\n w (x,0)|^2G(x,a)dx \le N \A \int w^2(x,0)G(x,a)dx.
\end{align}
In particular, by letting $\A_0=  N \operatorname{log}(2N(1+||V||_{\infty})\Theta_{\rho})+N(||V||_1^{1/2}+[V]_{1/2}^{1/2}+1)$ we deduce the following inequality
\begin{align}
		2a\int |\n w (x,0)|^2G(x,a)dx + \frac{n}{2} \int w^2(x,0)G(x,a)dx\le N^3\A_0 \int w^2(x,0)G(x,a)dx.
\end{align}
Now we use Lemma \ref{do} to obtain for all $r<1/2,$
\begin{align}\label{doi}
	\int_{B_{2r}} u^2(x,0)dx \le M \int_{B_r} u^2(x,0)dx,
\end{align}
where $M=\exp(N^4 \operatorname{log}(2N(1+||V||_{\infty})\Theta_{\rho})+N^4(||V||_1^{1/2}+[V]_{1/2}^{1/2}+1)).$\\
\textbf{Step 2}: We now show that there exists a universal constant $N$ such that for all $r<1/2,$
\begin{align}
	\int_{B_{2r}} u^2(x,0)dx \le M \int_{B_r} u^2(x,0)dx,
\end{align}  where $M= \exp({N \operatorname{log}(N\Theta)+N(||V||_1^{1/2}+[V]_{1/2}^{1/2}+1)})$ and $\Theta=N\int_{Q_4}u^2(x,t)dX\big/ \int_{B_1}u^2(x,0)dx.$\\

Note that in this step we want to replace $\Theta_{\rho}$ by $\Theta.$

\medskip

\textit{Proof of Step 2}:  For given $r\le \rho$, we can choose $k$ such that $$\rho \le 2^k r <2\rho.$$ Then using \eqref{doi} we have
\begin{align}
	\int_{B_{\rho}}u^2(x,0)dx \le \int_{B_{2^k r}}u^2(x,0)dx \le M^k\int_{B_r}u^2(x,0)dx \le M^{\operatorname{log}_2(2\rho/r)}\int_{B_r}u^2(x,0)dx.
\end{align}
We write $M=\exp(M_1+ N\operatorname{log}(N\Theta_{\rho})),$ where
\begin{align}\label{m1}
	M_1=N^4\operatorname{log}(2N(1+||V||_{\infty}))+N(||V||_1^{1/2}+[V]_{1/2}^{1/2}+1).
\end{align} Then we find 
\begin{align}
		\int_{B_{\rho}}u^2(x,0)dx \le e^{M_1\operatorname{log}_2(2\rho/r)}(N\Theta_{\rho})^{N\operatorname{log}_2(2\rho/r)} \int_{B_r}u^2(x,0)dx.
\end{align}
On putting the value of $\Theta_{\rho}$ we get 
\begin{align}\label{3sb}
		\int_{B_{\rho}}u^2(x,0)dx \le e^{M_1\operatorname{log}_2(2\rho/r)}\left(N\frac{\int_{Q_{4}} u^2( x, t)}{\rho^2 \int_{B_{\rho}}u^2( x,0)}\right)^{N\operatorname{log}_2(2\rho/r)} \int_{B_r}u^2(x,0)dx.
\end{align}
Since $\rho$ is a universal constant and less than 1 therefore $1/\rho^2$ can be absorbed into $N$ (by possibly taking  a larger $N$). Now we rearrange the terms in \eqref{3sb} to obtain
\begin{align}\label{3s}
		\int_{B_{\rho}}u^2(x,0)dx \le e^{M_1\frac{\operatorname{log}_2(2\rho/r)}{1+N\operatorname{log}_2(2\rho/r)}} \left(\int_{B_r}u^2(x,0)dx\right)^{\frac{1}{1+N\operatorname{log}_2(2\rho/r)}}\left(N\int_{Q_{4}} u^2( x, t)\right)^{\frac{N\operatorname{log}_2(2\rho/r)}{1+N\operatorname{log}_2(2\rho/r)}}.
\end{align}
In particular, if we take $r=\rho/2$, \eqref{3s} becomes 
\begin{align}\label{s1r}
	\int_{B_{\rho}}u^2(x,0)dx \le e^{M_1(1-\beta)} \left(\int_{B_{\rho/2}}u^2(x,0)dx\right)^{\beta}\left(N\int_{Q_{4}} u^2( x, t)\right)^{1-\beta},
\end{align}
where $\beta=\frac{1}{1+N\operatorname{log}_2 4}.$ Since the equation is translation invariant we can take translation of $u$ in space variable to obtain 
\begin{align}\label{3s1}
	\int_{B_{\rho}(y)}u^2(x,0)dx \le e^{M_1(1-\beta)} \left(\int_{B_{\rho/2}(y)}u^2(x,0)dx\right)^{\beta}\left(N\int_{Q_{4}} u^2( x, t)\right)^{1-\beta},
\end{align}
for all $y \in B_1.$ Note that $B_{3\rho/2} \subset \underset{y \in B_{\rho/2}}{\cup}B_{\rho}(y).$ From the compactness of $\overline{B_{5\rho/4}},$ there exist $y_1,y_2,...,y_{n_1} \in B_{\rho/2}$ where $n_1$ depends only on $n$  such that $B_{5\rho/4} \subset \underset{i}{\cup}B_{\rho}(y_i).$ Now we use \eqref{3s1} and sum over $i$ to find
\begin{align}
	\sum_{i=1}^{n_1}	\int_{B_{\rho}(y_i)}u^2(x,0)dx \le e^{M_1(1-\beta)}	\sum_{i=1}^{n_1} \left(\int_{B_{\rho/2}(y_i)}u^2(x,0)dx\right)^{\beta}\left(N\int_{Q_{4}} u^2( x, t)\right)^{1-\beta}.
\end{align}
Now in the left-hand side we use $B_{5\rho/4} \subset \underset{i}{\cup}B_{\rho}(y_i)$ and using $B_{\rho/2}(y_i) \subset B_{\rho}$ in the right-hand side to obtain
\begin{align}\label{3s2}
	\int_{B_{5\rho/4}}u^2(x,0)dx \le n_1e^{M_1(1-\beta)}	 \left(\int_{B_{\rho}}u^2(x,0)dx\right)^{\beta}\left(N\int_{Q_{4}} u^2( x, t)\right)^{1-\beta}.
\end{align}
We now use \eqref{s1r} in \eqref{3s2} to obtain 
\begin{align}\label{cov1}
	\int_{B_{5\rho/4}}u^2(x,0)dx &\le n_1e^{M_1(1-\beta)+M_1\beta(1-\beta)}	 \left(\int_{B_{\rho/2}}u^2(x,0)dx\right)^{\beta^2}\left(N\int_{Q_{4}} u^2( x, t)\right)^{1-\beta+\beta(1-\beta)}\\
	&=n_1e^{M_1(1-\beta^2)}	 \left(\int_{B_{\rho/2}}u^2(x,0)dx\right)^{\beta^2}\left(N\int_{Q_{4}} u^2( x, t)\right)^{1-\beta^2}. \notag
\end{align}
Proceeding in the same fashion, we can  get a similar inequality as in \eqref{cov1}  where on the left hand side,  the integral over $B_{5\rho/4}$  can replaced by integral over  $B_{7\rho/4}$ and so on.  Therefore since $\rho$ is universal ( in view of our choice in \eqref{cho}),  we can finally assert that there exists a  constant $N$ and $\tilde{\beta}\in (0,1)$ such that

\begin{align}
	\int_{B_1}u^2(x,0)dx &\le Ne^{M_1(1-\tilde{\beta})}	 \left(\int_{B_{\rho/2}}u^2(x,0)dx\right)^{\tilde{\beta}}\left(N\int_{Q_{4}} u^2( x, t)\right)^{1-\tilde{\beta}}.
\end{align}
This in particular implies
\begin{align}
	\int_{B_1}u^2(x,0)dx &\le Ne^{M_1(1-\tilde{\beta})}	 \left(\int_{B_{\rho}}u^2(x,0)dx\right)^{\tilde{\beta}}\left(N\int_{Q_{4}} u^2( x, t)\right)^{1-\tilde{\beta}}.
\end{align}
This can be rewritten as \begin{align}
	\int_{B_1}u^2(x,0)dx &\le Ne^{M_1(1-\tilde{\beta})}	 \left(\int_{B_{\rho}}u^2(x,0)dx\right)^{\tilde{\beta}}\left(N\int_{Q_{4}} u^2( x, t)\right)^{-\tilde{\beta}}\left(N\int_{Q_{4}} u^2( x, t)\right).
\end{align}
This gives $$\Theta_{\rho}^{\tilde{\beta}} \le  Ne^{M_1(1-\tilde{\beta})}\rho^{-2\tilde{\beta}}	\frac{N\int_{Q_{4}} u^2( x, t)}{\int_{B_{1}}u^2(x,0)dx}.$$ Hence using \eqref{doi}, we find
\begin{align}\label{dub}
	\int_{B_{2r}} u^2(x,0)dx \le M \int_{B_r} u^2(x,0)dx,
\end{align}
where $M=\exp(N^4 \operatorname{log}(2N(1+||V||_{\infty})N^2\rho^{-2}e^{M_1(1-\tilde{\beta})/\tilde{\beta}}\Theta^{1/{\tilde{\beta}}})+N(||V||_1^{1/2}+[V]_{1/2}^{1/2}+1)).$\\
Now  $M$ can be rewritten as $$M = \exp(N^4 \operatorname{log}(2N^3\rho^{-2}\Theta^{1/{\tilde{\beta}}})+\operatorname{log}(1+||V||_{\infty})+{M_1(1-\tilde{\beta})/\tilde{\beta}}+N(||V||_1^{1/2}+[V]_{1/2}^{1/2}+1)).$$ After recalling the value of $M_1$ from \eqref{m1} and using $\operatorname{log}(1+b) \le N b^{1/2}$ we get for a large enough $N$, $M$ in \eqref{dub} can be upper bounded in the following way
$$M\le \exp({N \operatorname{log}(N\Theta)+N(||V||_1^{1/2}+[V]_{1/2}^{1/2}+1)}).$$
\textbf{Step 3}: (Conclusion) We finally show that there exists a universal constant $N$ such that for all $r \le 1/2$ we have 
\begin{align}
	r^{\mathcal{C}} \le \int_{B_r} u^2(x,0)dx,
\end{align}
where $\mathcal{C}=\frac{1}{\int_{B_1} u^2(x,0)dx}+ N\operatorname{log}(N\Theta)+N(||V||_1^{1/2}+[V]_{1/2}^{1/2}+1)$ and  $\Theta$ is as in Step 2.\\
\textit{Proof of Step 3}: For $r<1/2$, choose $k\ge2$ such that $2^{-k} \le r \le 2^{-k+1}$ then using \eqref{dub} $k$-times 
\begin{align}
	\int_{B_1} u^2(x,0)dx \le M^k \int_{B_r} u^2(x,0)dx.
	\end{align}
Now $ r \le 2^{-k+1}$ gives $k \le \operatorname{log}_2(2/r).$ Hence $M^k \le M^{\operatorname{log}_2(2/r)}=\left(\frac{2}{r}\right)^{\frac{ \operatorname{ln} M}{\operatorname{ln}2}}.$ Thus we obtain 
\begin{align}\label{q1}
	\left(\frac{r}{2}\right)^{\frac{ \operatorname{ln} M}{\operatorname{ln}2}} \int_{B_1} u^2(x,0)dx \le \int_{B_r} u^2(x,0)dx.
\end{align}
Since for any $b$ we have $b \le 2^b,$ \eqref{q1} can be rewritten as
\begin{align}
	r^{\frac{ \operatorname{ln} M}{\operatorname{ln}2}} \le 2^{\left( \int_{B_1} u^2(x,0)dx \right)^{-1}}2^{\frac{ \operatorname{ln} M}{\operatorname{ln}2}}\int_{B_r} u^2(x,0)dx.
\end{align}
Multiply both side of \eqref{q1} by $r^{ \left(\int_{B_1} u^2(x,0)dx \right)^{-1}+\frac{ \operatorname{ln} M}{\operatorname{ln}2}}$ to obtain
\begin{align}
	r^{\left( \int_{B_1} u^2(x,0)dx \right)^{-1}+2\frac{ \operatorname{ln} M}{\operatorname{ln}2}} \le (2r)^{\left( \int_{B_1} u^2(x,0)dx \right)^{-1}+\frac{ \operatorname{ln} M}{\operatorname{ln}2}}\int_{B_r} u^2(x,0)dx.
\end{align}
Since $r\le 1/2$ therefore $(2r)^{\left( \int_{B_1} u^2(x,0)dx \right)^{-1}+\frac{ \operatorname{ln} M}{\operatorname{ln}2}} \le 1.$ Hence we get
\begin{align}
	r^{\mathcal{C}} \le \int_{B_r} u^2(x,0)dx,
\end{align}
where $\mathcal{C}=\frac{1}{\int_{B_1} u^2(x,0)dx}+ N\operatorname{log}(N\Theta)+N(||V||_1^{1/2}+[V]_{1/2}^{1/2}+1).$ This completes the proof of Theorem \ref{main}.
\end{proof}

\end{document}